\documentclass{jssc}

\def\textsubscript#1%
{$_{\text{#1}}$}
\newcommand*\supercite[1]{\textsuperscript{\cite{#1}}}
\def\cdd{\mbox{\boldmath$\cdot$}~}
\newcommand{\rulex}{\hfill\rule{1mm}{3mm}}
\input{amssym.def}

\usepackage{mathtools}
\usepackage{url}
\usepackage[algo2e,ruled,vlined]{algorithm2e}
\usepackage{dsfont}
\usepackage{color}
\usepackage{multirow,makecell}

\def\diag{{\rm diag}}
\def\rank{{\rm rank}}
\def\R{\mathcal{R}}
\def\M{\mathcal{M}}
\def\J{\mathcal{J}}
\def\I{\mathcal{I}}
\def\GL{{\rm GL}}
\def\K{\mathbb{K}}
\def\Aut{{\rm Aut}}
\def\TA{{\rm TA}}

\makeatletter
\def\@oddfoot{\hfill}
\newcount\shumeicount
\def\setshumei#1#2#3{%
  \shumeicount=\count0
  \def\@oddhead{%
    \raise-5pt\hbox to0pt{\vrule width\hsize height 0pt depth 0.4pt\hss}\relax
    \ifnum \shumeicount=\count0
      \raise-7pt\hbox to0pt{\vrule width\hsize height 0pt depth 0.4pt\hss}\relax
      #1
    \else
      \ifodd\count0
        #2
      \else
        #3
       \fi
     \fi
  }%
}
\makeatother
\makeatletter
\def\@oddfoot{\hfill}
\newcount\shujiaocount
\def\setshujiao{%
  \shujiaocount=\count0
  \def\@oddfoot{%
      \ifodd\count0
      \else
      \fi
  }%
}
\makeatother
\def\biaoti#1#2#3#4{{
  \vspace*{0.3cm}
  \begin{flushleft} \Large\bf #1\end{flushleft}
  \vspace*{-0.2cm}
      \begin{flushleft}
      \bf #2
      \end{flushleft}
      \footnotetext{\hspace{-6mm} #3\\ #4}}}
\def\dshm#1#2#3#4
{\setshumei{J Syst Sci Complex (#1) #2:
{\thepage--\pageref{LastPage}}\hfill}
            {\hfill {\small #3}\hfill\hbox to0pt{\hss\thepage}}
            {\hbox to0pt{\thepage\hss}\hfill {\small #4}\hfill
            }
            \setshujiao}
\def\drd#1#2
{{\vskip 1cm\small \begin{flushleft}
 #1 \\
 #2\\
\copyright The Editorial Office of JSSC \& Springer-Verlag GmbH Germany 2026
\end{flushleft}}}
\def\epsilon{\varepsilon}

\begin{document}
\include{graphix}
%*************************************************************************************************************
% \biaoti{THE CAPITALIZED TITLE OF YOUR ARTICLE$^*$}{The list of authors' names with the LAST NAME capitalized
% and the authors' names should be separated by "\cdd"}{the first author's name \\ the first author's affiliation
% and Email address\\ the second author's name\\ the second author's affiliation. More can be listed like this.}
% {$^*$ The titles and numbers of the foundations that support this article.}
%*************************************************************************************************************
\biaoti{Matrix equivalence to Smith normal form: new theoretical results for multivariate polynomial matrices$^*$}%%%   Main Title of your paper  %%%
{\uppercase{Lu} Dong $\cdd$ \uppercase{Ruan} Yuanyuan $\cdd$ \uppercase{Wang} Dingkang $\cdd$ \uppercase{Xiao} Fanghui}%%% The names of the authors  %%%
{\uppercase{Lu} Dong $\cdd$ \uppercase{Ruan} Yuanyuan \\
School of Mathematics, Southwest Jiaotong University, Chengdu 610031, China. Email\,$:$ donglu@swjtu.edu.cn; yuanyuanruan@my.swjtu.edu.cn.\\
\uppercase{Wang} Dingkang \\
State Key Laboratory of Mathematical Sciences, Academy of Mathematics and Systems Science, Chinese Academy of Sciences, Beijing 100190, China; School of Mathematical Sciences, University of Chinese Academy of Sciences, Beijing 100049, China. Email\,$:$ dwang@mmrc.iss.ac.cn. \\
\uppercase{Xiao} Fanghui (Corresponding author) \\
MOE-LCSM, School of Mathematics and Statistics, Hunan Normal University, Changsha 410081, China. Email\,$:$ xiaofanghui@hunnu.edu.cn.} %%% The address of the authors  %%%
{$^*$This research was supported by the National Key Research and Development Program of China under Grant No. 2025YFA1017200, the National Natural Science Foundation of China under Grant No. 12201210, the Sichuan Science and Technology Program under Grant No. 2024NSFSC0418.\\
{$^\diamond${\it This paper was recommended for publication by Editor  . }}}
%*************************************************************************************************************
%The submission date of your article. For example: \drd{Received: June 8, 2006}
%*************************************************************************************************************
\drd{DOI: }{Received: x x 2026}{ / Revised: x x 2026}

%*************************************************************************************************************
% The page header of the article.
% \dshm{Year}{Volume}{The capitalized RUNNING HEAD of your article with less than 48 letters}{The capitalized
% AUTHORS list with $\cdot$ separating different names or one can type "The name of the first author et al."
% if there are more than 4 authors.}
%*************************************************************************************************************

\dshm{2026}{XX}{NEW RESULTS ON MATRIX EQUIVALENCE}{\uppercase{Lu Dong} $\cdd$   \uppercase{Ruan Yuanyuan} $\cdd$ \uppercase{Wang Dingkang} $\cdd$ \uppercase{Xiao Fanghui}}
%*************************************************************************************************************
% \dab{The abstract}{Keywords}
%*************************************************************************************************************
%-------------------------------------------------------------------------
\Abstract{This paper investigates the Smith normal form equivalence problem for multivariate polynomial matrices. Using methods from matrix theory and polynomial ideal theory, we prove that Frost and Storey's 1978 conjecture holds for a broad class of matrices: such a matrix is equivalent to its Smith normal form if and only if its reduced minors of each order generate the unit ideal. Moreover, by extending the original matrix class via automorphisms of the polynomial ring, we show that our framework applies to a substantially more general setting.}      % the abstract

        % the keywords

%\MRSubClass{05B05, 05B25, 20B25}      % MR(2000) Subject Classification

%\baselineskip 15pt

\section{Introduction}

 Polynomial matrix theory serves as a foundational framework across diverse areas of mathematics and engineering, including symbolic computation\supercite{Cox2007}, algebraic topology\supercite{Noferini2025Smith}, and multidimensional systems theory\supercite{Bose1982,Bose2003}. A central problem in this field is determining when a multivariate polynomial matrix is equivalent to its Smith normal form, which is a canonical diagonal representation that preserves the rank, determinantal divisors and invariant factors of the original matrix.

 For univariate polynomial matrices over a field, the equivalence problem is completely resolved: since the ring of univariate polynomials is a principal ideal domain (PID), every such matrix can be transformed into its Smith normal form via elementary row and column operations\supercite{Gohberg1982Matrix}. However, this result fails to hold for multivariate polynomial rings (in two or more variables), which lack the PID structure. This fundamental difference has motivated decades of research into identifying conditions under which multivariate polynomial matrices admit Smith normal form equivalence.

 Frost and Storey\supercite{Frost1978} proposed a landmark conjecture asserting that a bivariate polynomial matrix is equivalent to its Smith normal form if and only if the reduced minors (see Definition \ref{def_reduced_minors}) of each order of the matrix generate the unit ideal in the polynomial ring. However, Frost and Storey\supercite{Frost1981} themselves constructed a counterexample demonstrating that the unit ideal condition is only necessary, not sufficient, in the general case. This counterexample prompted further work to characterize special classes of multivariate polynomial matrices for which the condition is necessary and sufficient.

 Over the past four decades, significant progress has been made in this direction. One well-studied class consists of square matrices over $\K[x_1,x_2,\ldots,x_n]$ whose determinant is of the form $(x_1 - f(x_2, \dots, x_n))^t$, where $\K$ is a field, $f\in \K[x_2,\ldots,x_n]$, and $n,t$ are integers with $n\geq 2$ and $t \geq 1$. Lin et al.\supercite{Lin2006On} first proved that for $t = 1$, any such matrix is always equivalent to its Smith normal form without additional conditions. Liu et al.\supercite{Liu2024} subsequently established that Frost and Storey's conjecture holds for all $t \geq 1$, showing that the unit ideal condition on reduced minors is necessary and sufficient in this general case. Furthermore, Liu et al.\supercite{Liu2025The} extended this theory to square matrices whose determinant is of the form $(x_1 - f_1(x_2, \dots, x_n))^{t_1}(x_2 - f_2(x_3, \dots, x_n))^{t_2}$, proving that the same equivalence characterization remains valid.

 Another important line of research concerns multivariate polynomial matrices whose highest-order determinantal divisor is a univariate polynomial. Li et al.\supercite{LiD2019} initiated this line of research by showing that any square bivariate polynomial matrix whose determinant is an irreducible univariate polynomial is always equivalent to its Smith normal form. Subsequently, Zheng et al.\supercite{Zheng2023New} proved that Frost and Storey's conjecture holds for any square matrix over $\K[x_1,x_2]$ with determinant of the form $p^t$, where $p\in \K[x_1]$ is irreducible and $t$ is a positive integer. Guan et al.\supercite{Guan2025NewR} further extended the main result in \cite{Zheng2023New} to square matrices over $\K[x_1,x_2,\ldots,x_n]$ of size at least $3$. Most recently, Lu et al.\supercite{Lu2024arxiv} developed a localization-based approach to completely resolve the problem for arbitrary square matrices over $\K[x_1,x_2,\ldots,x_n]$ with univariate determinant, demonstrating that the reduced minor condition remains valid in full generality.

 Following the discussion of the above two classes of matrices, recent research has focused on combining these two types into new families of matrices. For example, Zheng et al.\supercite{Zheng2024Biv} provided an equivalence characterization for square matrices over $\K[x_1,x_2]$ with determinant of the form $p(x_1)^{t_1}(x_2-f(x_1))^{t_2}$, where $t_1,t_2$ are nonnegative integers, $p\in \K[x_1]$ is irreducible, and $f\in \K[x_1]$ is arbitrary. Subsequently, Lu et al.\supercite{Lu2026Smith} developed a new method for investigating the Smith normal form equivalence problem for square matrices over $\K[x_1,x_2]$. Let $F$ be a square matrix over $\K[x_1,x_2]$, and let $\deg_{x_2}(\det(F))$ denote the degree of $\det(F)$ with respect to (w.r.t.) the variable $x_2$. By regarding $\det(F)$ as an element of $\K[x_1][x_2]$, they proved that Frost and Storey's conjecture holds when $\deg_{x_2}(\det(F)) \leq 1$, i.e., $\det(F)  =f(x_1) \cdot x_2 + g(x_1)$ for some $f,g\in \K[x_1]$, while counterexamples always exist demonstrating that the conjecture fails when $\deg_{x_2}(\det(F)) \geq 2$.

 Despite these significant advances, the general equivalence problem for multivariate polynomial matrices remains a fundamental open problem. Liu et al.\supercite{Liu2025The} raised the following open question: given a square matrix $F$ over $\K[x_1,x_2,\ldots,x_n]$ whose determinant has the form
 \begin{equation}\label{Intro-711-equ-1}
  (x_1-f_1(x_2,\ldots,x_n))^{t_1}\cdot (x_2-f_2(x_3,\ldots,x_n))^{t_2}\cdots (x_{n-1}-f_{n-1}(x_n))^{t_{n-1}}\cdot (x_n - \alpha)^{t_n},
 \end{equation}
 where $f_i\in \K[x_{i+1},\ldots,x_n]$ for $i=1,\ldots,n-1$, $\alpha\in \K$, and $t_1,\ldots,t_n$ are nonnegative integers, what are the necessary and sufficient conditions for $F$ to be equivalent to its Smith normal form? Building on the question posed by Liu et al.\supercite{Liu2025The}, we consider the following broader class of matrix equivalence problems.

\begin{problem}\label{main-problem}
 Let $F$ be a square matrix over $\K[x_1,x_2,\ldots,x_n]$ with full rank, and suppose
 \begin{equation}\label{Intro-711-equ-2}
  \det(F) = f_1(x_1)\cdot (x_2-f_2(x_1))^{t_2}\cdot(x_3 - f_3(x_1,x_2))^{t_3}
 \cdots (x_n-f_n(x_1,\ldots,x_{n-1}))^{t_n},
 \end{equation}
 where $f_1\in \K[x_1]$, $f_i\in \K[x_1,\ldots,x_{i-1}]$ for $i=2,\ldots,n$, and  $t_2,\ldots,t_n$ are nonnegative integers. Prove that $F$ is equivalent to its Smith normal form if and only if the reduced minors of each order of $F$ generate the unit ideal in $\K[x_1,x_2,\ldots,x_n]$.
\end{problem}

 To facilitate a direct comparison, we apply the variable reversal $x_i \mapsto x_{n-i+1}$ to Equation \eqref{Intro-711-equ-1} and relabel the indices of $f$ and $t$ accordingly. This yields an equivalent formulation, which we denote as the following equation:
 \begin{equation}\label{Intro-711-equ-3}
  (x_1-\alpha)^{t_1}\cdot (x_2-f_2(x_1))^{t_2}\cdot(x_3 - f_3(x_1,x_2))^{t_3}
 \cdots (x_n-f_n(x_1,\ldots,x_{n-1}))^{t_n}.
 \end{equation}
 In this form, the structural relationship between Equation \eqref{Intro-711-equ-1} and our target Equation \eqref{Intro-711-equ-2} becomes transparent. Since Equation \eqref{Intro-711-equ-2} reduces to Equation \eqref{Intro-711-equ-3} precisely when we set $f_1(x_1) = (x_1-\alpha)^{t_1}$, it follows that Equation \eqref{Intro-711-equ-2} naturally and strictly generalizes Equation \eqref{Intro-711-equ-3}, hence also Equation \eqref{Intro-711-equ-1} and the corresponding class of problems investigated in \cite{Liu2025The}. Moreover, we prove that Frost and Storey's conjecture holds in this more general setting.

 The rest of the paper is organized as follows. Section \ref{sec_Pre} introduces basic definitions and preliminary results on polynomial matrices, reduced minors, and matrix equivalence. Section \ref{sec_main-results} presents our main theorems and detailed proofs. Section \ref{sec_general} extends our results to more general settings, including non-square and rank-deficient matrices, and extensions under polynomial ring automorphisms. Section \ref{sec_conclusions} concludes the paper and outlines directions for future research.

\section{Preliminaries}\label{sec_Pre}

 Let $\mathbb{K}$ be a field with algebraic closure $\overline{\mathbb{K}}$, and let $\R = \K[x_1,\ldots,x_n]$ be the polynomial ring over $\K$ in the variables $x_1,\ldots,x_n$, where $n\geq 2$. Given $h_1,\ldots,h_l\in \R$, let $\langle h_1,\ldots,h_l \rangle_{\R}$ denote the ideal of $\R$ generated by $h_1,\ldots,h_l$, and let $\diag(h_1,\ldots,h_l)$ denote the $l\times l$ diagonal matrix with diagonal entries $h_1,\ldots,h_l$. For brevity, we define
 \[\R_i \triangleq \mathbb{K}[x_1,\ldots,x_{i-1},x_{i+1},\ldots,x_n] ~ \text{for} ~ i=2,\ldots,n,\]
 where each $\R_i$ is a polynomial ring over $\mathbb{K}$ in $n-1$ variables. Furthermore, define
 \[ \varphi_i \triangleq x_i - f_i(x_1,\ldots,x_{i-1})  ~ \mbox{for} ~ i=2,\ldots,n,\]
 where each $f_i\in \K[x_1,\ldots,x_{i-1}]$. Clearly, each $\varphi_i$ is irreducible in $\R$.

 We denote by $\mathcal{M}_{l\times m}(\R)$ the set of $l\times m$ matrices with entries in $\R$, where $l$ and $m$ are positive integers. Let $F\in \mathcal{M}_{l\times m}(\R)$. For each integer $i$ with $1\leq i \leq \min\{l,m\}$, let $\mathcal{I}_i(F)$ be the ideal of $\R$ generated by all $i\times i$ minors of $F$, and let $d_i(F)$ be the greatest common divisor of these minors. We call $d_i(F)$ the $i$-th determinantal divisor of $F$. By convention, we set $d_0(F) \equiv 1$.

\subsection{Basic Notions}

 Using the determinantal divisor defined above, we define the Smith normal form of a matrix over $\R$.

\begin{definition}
 Let $F \in \M_{l \times m}(\R)$ with rank $\gamma$, where $1\leq \gamma \leq \min\{l,m\}$. For each $i=1,\ldots,\gamma$, define the $i$-th invariant factor $h_i\in \R$ of $F$ by
 \[h_i \triangleq \frac{d_{i}(F)}{d_{i-1}(F)}.\]
 The Smith normal form of $F$ is given by
 \[S_F = \begin{pmatrix}
     {\rm diag}(h_1,\ldots, h_\gamma) &  0_{\gamma \times (m-\gamma)}  \\
         0_{(l-\gamma)\times \gamma}    &  0_{(l-\gamma)\times (m-\gamma)}
     \end{pmatrix}.\]
\end{definition}

 It is known that the invariant factors $h_1,\ldots,h_\gamma$ satisfy the divisibility relations $h_1\mid \cdots \mid h_\gamma$, as shown by Li et al.\supercite{Li2025Smith} using localization techniques. This shows that the Smith normal form for multivariate polynomial matrices constitutes a natural generalization of the Smith normal form for $\lambda$-matrices.

 We now recall the definition of unimodular matrices over $\R$, which will be used throughout the paper.

 \begin{definition}
  Let $U\in \M_{l\times l}(\R)$. Then $U$ is said to be unimodular if $\det(U)$ is a unit in $\R$. The set of all $l\times l$ unimodular matrices over $\R$ is denoted by $\GL_l(\R)$.
 \end{definition}

 Using unimodular matrices, we define matrix equivalence over $\R$.

 \begin{definition}
  Let $F,Q\in \M_{l\times m}(\R)$. We say that $F$ is equivalent to $Q$ over $\R$ if there exist $U\in \GL_l(\R)$ and $V\in \GL_m(\R)$ such that $UFV=Q$. The notation $F \sim_{\R} Q$ indicates that $F$ and $Q$ are equivalent over $\R$.
 \end{definition}

 To prove our main results on matrix equivalence, we introduce the notion of reduced minors, which will serve as our primary criterion for equivalence.

 \begin{definition}[see \cite{Lin1988}]\label{def_reduced_minors}
  Let $F \in \M_{l\times m}(\R)$ with rank $\gamma$, where $1\leq \gamma \leq \min\{l,m\}$. For any given integer $k$ with $1 \leq k \leq \gamma$, let $a^{(k)}_{1}, \ldots, a^{(k)}_{\beta_{k}}$ be all $k \times k$ minors of $F$, where $\beta_{k}=\binom{l}{k}\binom{m}{k}$. Extracting $d_{k}(F)$ from $a^{(k)}_{1}, \ldots, a^{(k)}_{\beta_{k}}$ yields
  \[ a^{(k)}_{j}=d_{k}(F) \cdot b^{(k)}_{j}, ~ j=1, \ldots, \beta_{k}.\]
  Then, $b^{(k)}_{1}, \ldots, b^{(k)}_{\beta_{k}}$ are called all $k \times k$ reduced minors of $F$. For convenience, we use $\mathcal{J}_k(F)$ to denote the ideal of $\R$ generated by $b^{(k)}_{1}, \ldots, b^{(k)}_{\beta_{k}}$.
 \end{definition}

 To compute minors of product matrices, which will be essential for analyzing the reduced minors of matrices, we need the classical Cauchy-Binet formula.

\begin{proposition}[Cauchy-Binet Formula\supercite{Strang2010Linear}]\label{binet-cauchy}
 Let $A\in \M_{l\times m}(\R)$, $B\in \M_{l\times t}(\R)$ and $C\in \M_{t\times m}(\R)$ satisfy $A = BC$, where $t$ is a positive integer. Then an $r\times r$ minor of $A$ is
 \begin{equation*}
  \det \left( A\begin{pmatrix}\begin{smallmatrix}i_1~\cdots ~i_r \\ j_1~\cdots ~j_r \end{smallmatrix}\end{pmatrix} \right)  =
  \sum_{1\leq s_1<\cdots<s_r\leq t}\det\left(B\begin{pmatrix}\begin{smallmatrix}
  i_1~\cdots~ i_r \\ s_1~\cdots~ s_r\end{smallmatrix}\end{pmatrix} \right)\cdot
  \det \left( C\begin{pmatrix}\begin{smallmatrix}s_1~\cdots ~s_r \\ j_1~\cdots ~j_r\end{smallmatrix}\end{pmatrix} \right),
 \end{equation*}
 where $1 \leq r \leq {\rm min}\{l,t,m\}$.
\end{proposition}

 Using the Cauchy-Binet formula, we can easily prove the following fundamental invariance property of matrix equivalence. We omit the detailed proof for brevity.

 \begin{proposition}\label{lemma-reduced-2}
  Let $A,B\in \M_{l\times m}(\R)$. If $A \sim_{\R} B$, then $\I_i(A)=\I_i(B)$, $d_i(A) = d_i(B)$, and $\J_i(A) = \J_i(B)$ for $i=1,\ldots,\min\{l,m\}$.
 \end{proposition}

 The following lemma describes the behavior of determinantal divisors and reduced minors under matrix equivalence, which is essential for proving our main results. This result was originally proved by Lu et al.\supercite{Lu2026Smith} for bivariate polynomial rings using the Cauchy-Binet formula, and we have verified that it holds for arbitrary multivariate polynomial rings.

 \begin{lemma}[see \cite{Lu2026Smith}]\label{Lu-2026-lemma-LAA}
  Let $A,B,C\in \M_{l\times l}(\R)$ satisfy $A \sim_{\R} BC$. If $\gcd(\det(B),\det(C))$ $= 1$, then
  \begin{enumerate}
    \item $d_i(A) = d_i(B)\cdot d_i(C)$ for $i=1,\ldots,l$;

    \item for $i=1,\ldots,l$, if $\J_i(A) = \R$, then $\J_i(B) = \J_i(C) = \R$.
  \end{enumerate}
 \end{lemma}

 The ideal-theoretic condition $\J_i(A) = \R$ appearing in Lemma \ref{Lu-2026-lemma-LAA} admits a natural geometric interpretation. To make this precise, we first recall the standard definition of an affine variety.

 \begin{definition}[see \cite{Cox2007}]
  Let $\mathcal{I} \subseteq \R$ be an ideal. Then the affine variety defined by $\mathcal{I}$ is
  \[\mathds{V}(\mathcal{I}) = \{\vec{\omega}\in \overline{\K}^{n} \mid h(\vec{\omega}) = 0 ~\text{for all}~ h\in \mathcal{I}\}.\]
 \end{definition}

 Li et al.\supercite{Li2022Li} established the fundamental correspondence between ideals and their associated affine varieties:
 \[\mathds{V}(\mathcal{I}) = \emptyset ~ \text{if and only if} ~ \mathcal{I} = \R.\]

 \subsection{Quillen-Suslin Theorem}

 To prepare for extending our main results (to be presented in subsequent sections) to the more general settings of non-square and rank-deficient matrices, we now introduce the key notion of zero prime matrices and recall the classical Quillen-Suslin theorem, a cornerstone of modern polynomial matrix theory.

 We begin with the definition of zero left prime and zero right prime matrices, which underpin the analysis of polynomial matrix factorization and equivalence.

 \begin{definition}[see \cite{Youla1979Notes}]
  Let $F \in \M_{l\times m}(\R)$ be of full row rank, where $l<m$. Then $F$ is said to be zero left prime (ZLP) if all $l \times l$ minors of $F$ generate the unit ideal in $\R$. Similarly, a matrix $F \in \M_{m\times l}(\R)$ of full column rank is said to be zero right prime (ZRP).
 \end{definition}

 The significance of zero prime matrices is highlighted by the Quillen-Suslin theorem, a landmark result obtained independently by Quillen\supercite{Quillen1976Projective} and Suslin\supercite{Suslin1976Projective} in their separate proofs of Serre's famous conjecture\supercite{Serre1955,Lam1978}. This theorem provides a canonical reduction for ZLP matrices via unimodular transformations.

 \begin{theorem}[Quillen-Suslin Theorem\supercite{Quillen1976Projective,Suslin1976Projective}]\label{QS-theorem}
  Let $F \in \M_{l \times m}(\R)$ be a ZLP matrix, where $l \leq m$. Then there exists $U\in \GL_m(\R)$ such that $FU = \left(\mathbf{I}_l, 0_{l\times(m-l)}\right)$, where $\mathbf{I}_l$ denotes the $l\times l$ identity matrix.
 \end{theorem}

 Building on the Quillen-Suslin theorem, Wang and Feng\supercite{Wang2004On} proved a generalization of Serre's conjecture proposed by Lin and Bose\supercite{Lin2001A}, which yields the following powerful full-rank factorization property for arbitrary-rank matrices. This result will be indispensable for our later generalization arguments.

 \begin{lemma}[Lin-Bose Lemma\supercite{Wang2004On}]\label{Lin-Bose-theorem}
  Let $F\in \M_{l\times m}(\R)$ with rank $\gamma$, where $1\leq \gamma \leq \min\{l,m\}$. If $\J_{\gamma}(F) = \R$, then there exist $G\in \M_{l\times \gamma}(\R)$ and $H\in \M_{\gamma\times m}(\R)$ such that $F = GH$ with $H$ being a ZLP matrix.
 \end{lemma}

 \subsection{Existing Equivalence Results}

 With the above tools in hand, we now recall several existing equivalence results for polynomial matrices. All of these results will serve as key building blocks for the proofs of our main theorems in the next section.

 We define a homomorphism modulo an irreducible polynomial in $\K[x_1]$, which will be used in the following lemma. Let $p\in \K[x_1]$ be irreducible. Then $\K[x_1]/\langle p \rangle$ is a field. Define the homomorphism
  \[\begin{array}{cccc}
    \phi_p:  & \R &  \longrightarrow & (\K[x_1]/\langle p \rangle)[x_2,\ldots,x_n] \\
           & \sum c_{i_2\cdots i_n}(x_1)\cdot x_2^{i_2}\cdots x_n^{i_n} & \longrightarrow & \sum \overline{c_{i_2\cdots i_n}(x_1)} \cdot x_2^{i_2}\cdots x_n^{i_n}, \\
   \end{array}\]
 where $c_{i_2\cdots i_n}(x_1)\in \K[x_1]$ and $\overline{c_{i_2\cdots i_n}(x_1)}\in \K[x_1]/\langle p \rangle$. This extends entry-wise to a homomorphism $\phi_p: \M_{l\times m}(\R) \rightarrow \M_{l\times m}((\K[x_1]/\langle p \rangle)[x_2,\ldots,x_n])$.
 
 The next two lemmas extend Lemmas 3.3 and 3.6 of \citep{Lu2025JSSC} from $\K[x_1,x_2]$ to $\R$. Their proofs follow by similar rank and Laplace-expansion arguments, so we omit them.

 \begin{lemma}\label{lemma-25JSSC-1}
  Let $F\in \M_{l\times m}(\R)$ with $l\leq m$, and let $p\in \K[x_1]$ be irreducible. Suppose there exists a $k\times m$ submatrix $F_1$ of $F$ such that $\rank(\phi_p(F_1))=k$, where $k$ is an integer with $1\leq k \leq l-1$. If for any row vector $\vec{u}\in F \setminus F_1$, the matrix $F_2^{(\vec{u})} = \begin{pmatrix} F_1 \\ \vec{u}\end{pmatrix} \in \M_{(k+1)\times m}(\R)$ satisfies $\rank(\phi_p(F_2^{(\vec{u})})) = k$, then $\rank(\phi_p(F))=k$.
 \end{lemma}

% The proof is analogous to that of Lemma 3.3 in \citep{Lu2025JSSC} and is omitted here. {\color{red}{(need revise)}}
%
% The following lemma extends Lemma 3.6 in \cite{Lu2025JSSC} from $K[x_1,x_2]$ to $\R$. Its proof uses the same Laplace-expansion argument and is omitted.

 \begin{lemma}\label{422-lemma-1}
  Let $A\in \M_{l\times l}(\R)$ be of full rank, and let $h_1,\ldots,h_k,h\in \R$ satisfy $h_1\mid \cdots \mid h_k \mid h$, where $k$ is an integer with $1\leq k \leq l-1$. If there exists $B\in \M_{(l-k)\times (l-k)}(\R)$ such that
  \[ A = \diag(h_1,\ldots,h_k,\underbrace{h,\ldots,h}_{l-k}) \cdot \begin{pmatrix}
      \mathbf{I}_k &  \\  & B \end{pmatrix},\]
  then
  \[ d_i(B) = \frac{d_{k+i}(A)}{h_1\cdots h_k h^i} ~ \text{and} ~ \J_i(B) = \J_{k+i}(A),  ~ \text{where} ~ i=1,\ldots,l-k.\]
 \end{lemma}

% \begin{lemma}\label{414-lem-1}
%  Let $A\in \M_{l\times l}(\R)$ and $U\in \GL_l(\R_r)$, where $r$ is an integer with $2\leq r \leq n$. Let $h_1,\ldots,h_l\in \R$ satisfy $h_1\mid \cdots \mid h_l$. If there exist integers $t_1,\ldots,t_l$ with $0\leq t_1\leq \cdots\leq t_l$ such that
%  \[ A = \diag(h_1,\ldots,h_l) \cdot U \cdot \diag(\varphi_r^{t_1},\ldots,\varphi_r^{t_l}),\]
%  then the Smith normal form of $A$ is $S_A = \diag(h_1\varphi_r^{t_1},\ldots,h_l\varphi_r^{t_l})$.
% \end{lemma}

 The following three lemmas, due to \cite{Lu2024arxiv}, use localization techniques to characterize such equivalences, first for square matrices over $\R$ with determinants in $\K[x_1]$, then extending to non-square and rank-deficient matrices via the Quillen-Suslin theorem.

\begin{lemma}[see \cite{Lu2024arxiv}]\label{Zheng-lemma}
 Let $F\in \M_{l\times l}(\R)$ with $p\mid \det(F)$, where $p\in \K[x_1]$ is irreducible. Suppose there exists an integer $k ~ (1\leq k \leq l-1)$ satisfying $\J_k(F)=\R$ and $\rank(\phi_p(F)) = k$. Then $F$ can be factorized as
 \[ F = U \cdot  \diag(\underbrace{1,\ldots,1}_{k},p,\ldots,p) \cdot G,\]
 where $U\in \GL_l(\R)$ and $G\in \M_{l\times l}(\R)$.
\end{lemma}

 \begin{lemma}[see \cite{Lu2024arxiv}]\label{424-lemma-1}
  Let $A\in \M_{l\times l}(\R)$ with $p\mid \det(A)$, and $U\in GL_l(\R)$, where $p\in \K[x_1]$ is irreducible. Suppose there exists an integer $k$ with $1\leq k \leq l-1$ such that
  \[A = \diag(p^{s_1},\ldots,p^{s_k},p^{s},\ldots,p^{s}) \cdot U \cdot \diag(\underbrace{1,\ldots,1}_{k},p,\ldots,p),\]
  where $s_1,\ldots,s_k,s$ are integers satisfying $0\leq s_1 \leq \cdots \leq s_k \leq s$. If $d_i(A) = p^{s_1+\cdots+s_i}$ and $\J_i(A) = \R$ for $i=1,\ldots,k$, then
  \[A\sim_{\R} \diag(p^{s_1},\ldots,p^{s_k},p^{s+1},\ldots,p^{s+1}).\]
 \end{lemma}

 \begin{lemma}[see \cite{Lu2024arxiv}]\label{415-lemma-1}
  Let $F\in \M_{l\times m}(\R)$ with $d_{\gamma}(F)\in \K[x_1]$, where $\gamma$ is the rank of $F$ with $1\leq \gamma \leq \min\{l,m\}$. Then $F$ is equivalent over $\R$ to its Smith normal form if and only if $\J_i(F) = \R$ for $i=1,\ldots,\gamma$.
 \end{lemma}

 Liu et al.\supercite{Liu2024} investigated the equivalence of square matrices to their Smith normal forms when the determinant is a power of $x_1-f(x_2,\ldots,x_n)$. Adapting their arguments to our setting yields the following three results for square matrices whose determinants are powers of $\varphi_r$.

 \begin{lemma}\label{424-lemma-2}
  Let $F\in \M_{l\times l}(\R)$ with $\varphi_r\mid \det(F)$, where $r$ is an integer with $2\leq r \leq n$. Suppose there exists an integer $k ~ (1\leq k \leq l-1)$ satisfying $\J_k(F) = \R$ and $\rank(F(x_1,\ldots,x_{r-1},f_r,x_{r+1},\ldots,x_n))$ $= k$. Then $F$ can be factorized as
  \[ F = V \cdot \diag(\underbrace{1,\ldots,1}_{k}, \varphi_r,\ldots,\varphi_r) \cdot H,\]
  where $V\in \GL_l(\R_r)$ and $H\in \M_{l\times l}(\R)$.
 \end{lemma}

 \begin{lemma}\label{502-lemma-1}
  Let $A\in \M_{l\times l}(\R)$ with $\varphi_r\mid \det(A)$, and $U\in GL_l(\R_r)$, where $r$ is an integer with $2\leq r \leq n$. Suppose there exists an integer $k$ with $1\leq k \leq l-1$ such that
  \[A = \diag(\varphi_r^{s_1},\ldots,\varphi_r^{s_k},
  \varphi_r^{s},\ldots,\varphi_r^{s}) \cdot U \cdot \diag(\underbrace{1,\ldots,1}_{k},\varphi_r,\ldots,\varphi_r),\]
  where $s_1,\ldots,s_k,s$ are integers satisfying $0\leq s_1 \leq \cdots \leq s_k \leq s$. If $d_i(A) = \varphi_r^{s_1+\cdots+s_i}$ and $\J_i(A) = \R$ for $i=1,\ldots,k$, then
  \[A\sim_{\R} \diag(\varphi_r^{s_1},\ldots,\varphi_r^{s_k},\varphi_r^{s+1},
  \ldots,\varphi_r^{s+1}).\]
 \end{lemma}

 \begin{lemma}\label{425-lemma-1}
  Let $F\in \M_{l\times m}(\R)$ with $d_{\gamma}(F) = \varphi_r^t$, where $\gamma$ is the rank of $F$ with $1\leq \gamma \leq \min\{l,m\}$, $r,t$ are integers satisfying $2\leq r \leq n$ and $t \geq 1$. Then $F$ is equivalent over $\R$ to its Smith normal form if and only if $\J_i(F) = \R$ for $i=1,\ldots,\gamma$.
 \end{lemma}

 Lemmas \ref{Zheng-lemma}--\ref{415-lemma-1} and \ref{424-lemma-2}--\ref{425-lemma-1} establish parallel results for two distinct classes of polynomial matrices over $\R$: the former concerns matrices with determinants involving univariate polynomials, while the latter addresses matrices whose determinants are powers of $\varphi_r ~ (2\leq r \leq n)$. In the next section, we combine these two settings and establish a general equivalence theorem for square matrices over $\R$ whose determinants contain both types of factors.

\section{Main Results}\label{sec_main-results}

 With all the preliminary tools and known results in place, we now turn to the main contributions of this paper. Our core task is to address Problem \ref{main-problem} and establish our main equivalence theorem (Theorem \ref{425-Theorem-2}). To this end, we begin by introducing a new definition that will play a key role in our subsequent arguments.

 \begin{definition}
  Let $F\in \M_{l\times l}(\R)$ be of full rank, and let $p_1^{r_1}p_2^{r_2}\cdots p_t^{r_t}$ be the irreducible factorization of $\det(F)$, where $p_1,p_2,\ldots,p_t\in \R$ are distinct and irreducible, $t$ and $r_1,r_2,\ldots,r_t$ are positive integers. Let the Smith normal form of $F$ be
  \begin{equation*}
   \left(
         \begin{array}{cccc}
           p_1^{s_{11}}p_2^{s_{12}}\cdots p_t^{s_{1t}} & &  & \\  &  p_1^{s_{21}}p_2^{s_{22}}\cdots p_t^{s_{2t}} &  & \\
           &  & \ddots & \\
            &  &  & p_1^{s_{l1}}p_2^{s_{l2}}\cdots p_t^{s_{lt}} \\
         \end{array}
       \right),
  \end{equation*}
  where for $j=1,\ldots,t$, $s_{1j},s_{2j},\ldots,s_{lj}$ are integers satisfying $0\leq s_{1j} \leq s_{2j} \leq \cdots \leq s_{lj}$ and $r_j = \sum_{i=1}^{l}s_{ij}$. Then
  \[\diag(p_j^{s_{1j}},p_j^{s_{2j}},\ldots,p_j^{s_{lj}})\]
  is called the Smith normal form of $F$ w.r.t. $p_j$, where $j=1,\ldots,t$.
 \end{definition}

 With this definition in hand, we now state the following lemma.

 \begin{lemma}\label{419-lemma-1}
  Let $F\in \M_{l\times l}(\R)$, and let $p\in \K[x_1]$ be an irreducible factor of $\det(F)$. Let $\diag(p^{s_1},\ldots,p^{s_l})$ be the Smith normal form of $F$ w.r.t. $p$, where $s_1,\ldots,s_l$ are integers satisfying $0\leq s_1\leq \cdots \leq s_l$. Suppose there exist integers $k ~ (1\leq k \leq l-1)$ and $s ~ (s_k \leq s < s_{k+1})$ such that
  \[ F \sim_{\R} \diag(p^{s_1},\ldots,p^{s_k},p^s,\ldots,p^s) \cdot G,\]
  where $G\in \M_{l\times l}(\R)$. If $\J_i(F) = \R$ for $i=1,\ldots,l$, then
  \begin{enumerate}
    \item $d_i(G) = \frac{d_i(F)}{p^{s_1+\cdots+s_i}}$ and $\J_i(G) = \R$ for $i=1,\ldots,k$;

    \item $\rank(\phi_p(G)) = k$.
  \end{enumerate}
 \end{lemma}

 \begin{proof}
  Let $e_i = \sum_{j=1}^{i}s_{j}$, where $i=1,\ldots,k$. Since $\diag(p^{s_{1}},\ldots,p^{s_{l}})$ is the Smith normal form of $F$ w.r.t. $p$, there exists $q_i\in \R$ such that
  \begin{equation}\label{Lu_equ-1}
   d_i(F) = p^{e_i}q_i ~ \text{and} ~ \gcd(p,q_i) = 1 ~ \text{for} ~ i=1,\ldots,k.
  \end{equation}
  Let $A = \diag(p^{s_{1}},\ldots,p^{s_{k}},p^{s},
  \ldots,p^{s}) \cdot G$. Since $F\sim_{\R} A$, it follows from Proposition \ref{lemma-reduced-2} that
  \begin{equation}\label{Lu_equ-2}
   d_i(A) = d_i(F) \text{ and } \J_i(A) = \J_i(F) \text { for } i=1,\ldots,k.
  \end{equation}
  For any given integer $i_0$ with $1 \leq i_0 \leq k$, assume that $h_{11}^{(i_0)},\ldots,h_{1\eta_{i_0}}^{(i_0)},h_{21}^{(i_0)},\ldots,h_{2\xi_{i_0}}^{(i_0)}\in \R$ are all $i_0\times i_0$ minors of $G$, where $h_{11}^{(i_0)},\ldots,h_{1\eta_{i_0}}^{(i_0)}$ are all $i_0\times i_0$ minors of the submatrix formed by the first $i_0$ rows of $G$. Then,
  \[ p^{e_{i_0}}h_{11}^{(i_0)},\ldots,p^{e_{i_0}}h_{1\eta_{i_0}}^{(i_0)},
  p^{e_{21}^{(i_0)}}h_{21}^{(i_0)},\ldots,p^{e_{2\xi_{i_0}}^{(i_0)}}h_{2\xi_{i_0}}^{(i_0)}\]
  are all $i_0\times i_0$ minors of $A$, where $e_{2j}^{(i_0)} = s_{j_1}+\cdots+s_{j_{i_0}}$, the indices $\{j_1,\ldots,j_{i_0}\}$ is a strictly increasing sequence with $1\leq j_1 < \cdots < j_{i_0} \leq l$, $s_{j_{i_0}} = s$ if $j_{i_0} \geq k+1$, $j=1,\ldots,\xi_{i_0}$. Clearly, $e_{2j}^{(i_0)} \geq e_{i_0}$ for $j=1,\ldots,\xi_{i_0}$. Thus,
  \begin{equation}\label{Lu_equ-3}
   \begin{split}
   d_{i_0}(A) & = \gcd\left(p^{e_{i_0}}h_{11}^{(i_0)},\ldots,p^{e_{i_0}}h_{1\eta_{i_0}}^{(i_0)},
  p^{e_{21}^{(i_0)}}h_{21}^{(i_0)},\ldots,p^{e^{(i_0)}_{2\xi_{i_0}}}h_{2\xi_{i_0}}^{(i_0)}\right) \\ & =  p^{e_{i_0}} \cdot \gcd\left(h_{11}^{(i_0)},\ldots,h_{1\eta_{i_0}}^{(i_0)},
  p^{e^{(i_0)}_{21}-e_{i_0}}h_{21}^{(i_0)},\ldots,p^{e^{(i_0)}_{2\xi_{i_0}}-e_{i_0}}h_{2\xi_{i_0}}^{(i_0)}\right).
   \end{split}
  \end{equation}
  Combining Equations \eqref{Lu_equ-1}--\eqref{Lu_equ-3}, we have
  \begin{equation*}\label{Lu_equ-4}
   q_{i_0} = \gcd\left(h_{11}^{(i_0)},\ldots,h_{1\eta_{i_0}}^{(i_0)},
  p^{e^{(i_0)}_{21}-e_{i_0}}h_{21}^{(i_0)},\ldots,p^{e^{(i_0)}_{2\xi_{i_0}}-e_{i_0}}h_{2\xi_{i_0}}^{(i_0)}\right).
  \end{equation*}
  It follows from $q_{i_0} \mid p^{e^{(i_0)}_{2j}-e_{i_0}}h_{2j}^{(i_0)}$ and $\gcd(p,q_{i_0}) = 1$ that $q_{i_0} \mid h^{(i_0)}_{2j}$ for $j=1,\ldots,\xi_{i_0}$. This implies that
  \begin{equation*}\label{Lu_equ-5}
   q_{i_0} \mid \gcd\left(h_{11}^{(i_0)},\ldots,h_{1\eta_{i_0}}^{(i_0)},h_{21}^{(i_0)},\ldots,h_{2\xi_{i_0}}^{(i_0)}\right).
  \end{equation*}
  Since  \[\gcd\left(h_{11}^{(i_0)},\ldots,h_{1\eta_{i_0}}^{(i_0)},h_{21}^{(i_0)},\ldots,h_{2\xi_{i_0}}^{(i_0)}\right) \bigm| \gcd\left(h_{11}^{(i_0)},\ldots,h_{1\eta_{i_0}}^{(i_0)},
  p^{e^{(i_0)}_{21}-e_{i_0}}h_{21}^{(i_0)},\ldots,p^{e^{(i_0)}_{2\xi_{i_0}}-e_{i_0}}h_{2\xi_{i_0}}^{(i_0)}\right),\] we get
  \begin{equation*}\label{Lu_equ-5-1}
   q_{i_0} = \gcd\left(h_{11}^{(i_0)},\ldots,h_{1\eta_{i_0}}^{(i_0)},h_{21}^{(i_0)},\ldots,h_{2\xi_{i_0}}^{(i_0)}\right) = d_{i_0}(G).
  \end{equation*}
  Then
  \[\frac{h_{11}^{(i_0)}}{q_{i_0}},\ldots,\frac{h_{1\eta_{i_0}}^{(i_0)}}{q_{i_0}},
  p^{e^{(i_0)}_{21}-e_{i_0}}\frac{h_{21}^{(i_0)}}{q_{i_0}},\ldots,
  p^{e^{(i_0)}_{2\xi_{i_0}}-e_{i_0}}\frac{h_{2\xi_{i_0}}^{(i_0)}}{q_{i_0}}\]
  are all $i_0\times i_0$ reduced minors of $A$. It follows from $\J_{i_0}(A) = \R$ that
  \begin{equation*}\label{Lu_equ-5-2}
   \J_{i_0}(A) \subseteq \left\langle \frac{h_{11}^{(i_0)}}{q_{i_0}},\ldots,\frac{h_{1\eta_{i_0}}^{(i_0)}}{q_{i_0}},
  \frac{h_{21}^{(i_0)}}{q_{i_0}},\ldots,\frac{h_{2\xi_{i_0}}^{(i_0)}}{q_{i_0}} \right\rangle_{\R} = \R.
  \end{equation*}
  As $i_0$ is an arbitrarily chosen integer in $\{1,\ldots,k\}$, we conclude that
  \begin{equation*}\label{Lu_equ-5-3}
   d_i(G) = \frac{d_i(F)}{p^{s_1+\cdots+s_i}} ~ \text{and} ~ \J_i(G) = \R ~ \text{for} ~ i=1,\ldots,k.
  \end{equation*}

  In what follows, we divide the proof into three cases to show that $\rank(\phi_p(G)) = k$.

  \textbf{First Case:} $s_1=\cdots=s_k = s$.

  Since $d_k(G) = q_k$ and $\gcd(p,q_k)=1$, we have $\rank(\phi_p(G))\geq k$. For any given two strictly increasing sequences $\{i_1,\ldots,i_{k+1}\}$ with $1\leq i_1 < \cdots < i_{k+1}\leq l$ and $\{j_1,\ldots,j_{k+1}\}$ with $1\leq j_1 < \cdots < j_{k+1}\leq l$, we get
  \[p^{ks+s_{k+1}} \bigm| \det\left(A\begin{pmatrix} i_1 \cdots i_{k+1} \\ j_1 \cdots j_{k+1}
  \end{pmatrix}\right) = p^{(k+1)s}\cdot \det\left(G\begin{pmatrix} i_1 \cdots i_{k+1} \\ j_1 \cdots j_{k+1} \end{pmatrix}\right).\]
  As $s_{k+1}>s$, we obtain
  \[p \bigm| \det\left(G\begin{pmatrix} i_1 \cdots i_{k+1} \\ j_1 \cdots j_{k+1} \end{pmatrix}\right).\]
  This implies that $p\mid d_{k+1}(G)$. Thus, $\rank(\phi_p(G))= k$.

  \textbf{Second Case:} $s_k < s$.

  Let $G_k$ be the $k\times l$ submatrix formed by the first $k$ rows of $G$. We assert that $\rank(\phi_p(G_k)) = k$. If otherwise, $p\mid h_{1j}^{(k)}$ for $j=1,\ldots,\eta_k$. Since $s_k < s$, we have $e_{2j}^{(k)}> e_k$ for $j=1,\ldots,\xi_k$. It follows that
  \begin{equation*}\label{Lu_equ-6}
   p \bigm| \gcd\left(h_{11}^{(k)},\ldots,h_{1\eta_{k}}^{(k)},
  p^{e^{(k)}_{21}-e_{k}}h_{21}^{(k)},\ldots,p^{e^{(k)}_{2\xi_{k}}-e_{k}}h_{2\xi_{k}}^{(k)}\right).
  \end{equation*}
  This implies that $p\mid q_k$, which contradicts the fact that $\gcd(p,q_k) = 1$. Thus, $\rank(\phi_p(G_k)) = k$. Let $G_{k,j}$ be the $(k+1)\times l$ submatrix formed by the first $k$ rows and the $j$-th row of $G$, where $j=k+1,\ldots,l$. We assert that $\rank(\phi_p(G_{k,j})) = k$ for $j=k+1,\ldots,l$. Without loss of generality, we consider the case $j=k+1$. As $G_k$ is a submatrix of $G_{k,k+1}$, we have $\rank(\phi_p(G_{k,k+1})) \geq k$. For any given strictly increasing sequence $1\leq i_1 < \cdots < i_{k+1} \leq l$, the following determinant
  \[p^{e_k+s} \cdot \det\left( G_{k,k+1} \begin{pmatrix} ~ 1 ~ \cdots ~ k ~ k+1 \\ i_1 ~ \cdots ~ i_k ~ i_{k+1} \end{pmatrix}\right)\]
  is a $(k+1)\times (k+1)$ minor of $A$. Since $p^{e_k+s_{k+1}} \mid d_{k+1}(A)$ and $s<s_{k+1}$, we obtain
  \[p \bigm| \det\left( G_{k,k+1} \begin{pmatrix} ~ 1 ~ \cdots ~ k ~ k+1 \\ i_1 ~ \cdots ~ i_k ~ i_{k+1} \end{pmatrix}\right).\]
  It follows that $p$ divides all $(k+1)\times (k+1)$ minors of $G_{k,k+1}$. Consequently, $\rank(\phi_p(G_{k,k+1})) = k$. According to Lemma \ref{lemma-25JSSC-1}, $\rank(\phi_p(G)) = k$.

  \textbf{Third Case:} there exists an integer $\tau$ with $1\leq \tau \leq k-1$ such that
  \[s_1\leq \cdots \leq s_{\tau} < s_{\tau+1}  = \cdots = s_k = s.\]

  We assert that there exists a strictly increasing sequence $\tau+1 \leq i_1 < \cdots <i_{k - \tau} \leq l$, the $k\times l$ submatrix $G_k^{(\vec{\tau})}$ formed by the first $\tau$ rows and the $i_1$-th, $\ldots$, $i_{k - \tau}$-th rows of $G$ satisfies $\rank(\phi_p(G_k^{(\vec{\tau})}))=k$, where $\vec{\tau}=(i_1,\ldots,i_{k - \tau})$. If otherwise, $p$ divides all $k\times k$ minors of $G_k^{(\vec{\tau})}$. Let $A_k^{(\vec{\tau})}$ be the $k\times l$ submatrix formed by the first $\tau$ rows and the $i_1$-th, $\ldots$, $i_{k - \tau}$-th rows of $A$. Then
  \[p^{e_{\tau}+(k-\tau)s+1} \mid d_k(A_k^{(\vec{\tau})}).\]
   Let $G = (g_{ij})_{l\times l}$, where $g_{ij}\in \R$ for $1\leq i,j \leq l$. Then
  \[ A = \begin{pmatrix}
     p^{s_1}g_{11} & \cdots & p^{s_1}g_{1\tau} & p^{s_1}g_{1,\tau+1} & \cdots & p^{s_1}g_{1l} \\
     \vdots & \ddots & \vdots & \vdots & \ddots & \vdots \\
     p^{s_\tau}g_{\tau1} & \cdots & p^{s_\tau}g_{\tau\tau} & p^{s_\tau}g_{\tau,\tau+1} & \cdots & p^{s_\tau}g_{\tau l} \\
     p^{s}g_{\tau+1,1} & \cdots & p^{s}g_{\tau+1,\tau} & p^{s}g_{\tau+1,\tau+1} & \cdots & p^{s}g_{\tau+1,l} \\
     \vdots & \ddots & \vdots & \vdots & \ddots & \vdots \\
     p^{s}g_{l1} & \cdots & p^{s}g_{l\tau} & p^{s}g_{l,\tau+1} & \cdots & p^{s}g_{ll}
  \end{pmatrix}.\]
  For any given two strictly increasing sequences $\{j_1,\ldots,j_k\}$ with $1\leq j_1< \cdots < j_k \leq l$ and $\{t_1,\ldots,t_k\}$ with $1\leq t_1< \cdots < t_k \leq l$, we have
  \[p^{s_{j_1}+\cdots+s_{j_k}} \bigm| \det\left(A\begin{pmatrix} j_1\cdots j_k \\ t_1 \cdots t_k\end{pmatrix}\right).\]
  If $(j_1,\ldots,j_{\tau}) \neq (1,\ldots,\tau)$, then $s_{j_1}+\cdots+s_{j_k} > e_{\tau}+(k-\tau)s$. This implies that
  \[p^{e_{\tau}+(k-\tau)s+1} \mid d_k(A).\]
  Since $d_k(A) = p^{e_{\tau}+(k-\tau)s}q_k$, we have $p\mid q_k$. This contradicts the fact that $\gcd(p,q_k)=1$. Therefore, $\rank(\phi_p(G_k^{(\vec{\tau})}))=k$. For any given integer $\delta\in \{1,\ldots,l\} \setminus \{1,\ldots,\tau,i_1,\ldots,i_{k-\tau}\}$, let $G_{k,\delta}^{(\vec{\tau})}$ be the $(k+1)\times l$ submatrix formed by $G_{k}^{(\vec{\tau})}$ and the $\delta$-th row of $G$. Since $p^{e_{\tau}+(k-\tau)s+s_{k+1}}\mid d_{k+1}(A)$, we get
  \[p^{e_{\tau}+(k-\tau)s+s_{k+1}}\mid p^{e_{\tau}+(k+1-\tau)s}d_{k+1}(G_{k,\delta}^{(\vec{\tau})}).\]
  By the fact that $s_{k+1}>s$, we can draw the conclusion that $p\mid d_{k+1}(G_{k,\delta}^{(\vec{\tau})})$. Based on Lemma \ref{lemma-25JSSC-1}, $\rank(\phi_p(G)) = k$. \rulex
 \end{proof}

 It is a classical result that every full-rank polynomial matrix over $\K[x_1,x_2]$ admits a primitive factorization\supercite{Guiver1982}. The following lemma establishes an analogous factorization property over $\R$, while requiring an additional technical condition.

 \begin{lemma}\label{423-lemma-1}
  Let $F\in \M_{l\times l}(\R)$ with $\det(F) = gh$, where $g\in \K[x_1]$ and $h\in \R$ satisfy $\gcd(g,h)=1$. If $\J_i(F)=\R$ for $i=1,\ldots,l$, then there exist $G,H\in \M_{l\times l}(\R)$ such that
  \[F = GH ~ \text{and} ~ \det(G) = g.\]
 \end{lemma}

 \begin{proof}
  Let $g = \alpha p_1^{t_1}\cdots p_r^{t_r}$ be an irreducible factorization of $g$, where $\alpha\in \K \setminus \{0\}$, $p_1,\ldots,p_r\in \K[x_1]$ are pairwise coprime irreducible factors, $r$ and $t_1,\ldots,t_r$ are positive integers. Without loss of generality, assume that the Smith normal form of $F$ w.r.t. $p_1$ is
  \begin{equation*}\label{427-lem-equ-01}
   S_{p_1} = \diag(p_1^{s_{11}}, p_1^{s_{21}}, \ldots,p_1^{s_{l1}}),
  \end{equation*}
  where $s_{11},s_{21},\ldots, s_{l1}$ are integers satisfying $0\leq s_{11} \leq s_{21} \leq \cdots \leq s_{l1}$ and $t_1 = \sum_{j=1}^{l}s_{j1}$. Since $p_1^{s_{11}}\mid d_1(F)$, we may factor out $p_1^{s_{11}}$ from $F$ to obtain
  \begin{equation*}\label{427-lem-equ-02}
   F = \diag(p_1^{s_{11}}, p_1^{s_{11}}, \ldots,p_1^{s_{11}}) \cdot F_1,
  \end{equation*}
  where $F_1\in \M_{l\times l}(\R)$. If $s_{11} = s_{21}$, then we consider the order relation between $s_{21}$ and $s_{31}$. Otherwise, by Lemma \ref{419-lemma-1} we have $d_1(F_1) = \frac{d_1(F)}{p_1^{s_{11}}}$, $\J_1(F_1)=\R$ and $\rank(\phi_{p_1}(F_1)) = 1$. According to Lemma \ref{Zheng-lemma}, there exist $U_1\in \GL_l(\R)$ and $F_2\in \M_{l\times l}(\R)$ such that
  \begin{equation*}\label{427-lem-equ-03}
   F_1 = U_1 \cdot \diag(1,\underbrace{p_1,\ldots,p_1}_{l-1}) \cdot F_2.
  \end{equation*}
  Adopting the same proof argument as in the first part of Lemma \ref{419-lemma-1}, we deduce that $d_1(F_2) = d_1(F_1)$. Let
  \begin{equation*}\label{427-lem-equ-04}
   A = \diag(p_1^{s_{11}}, p_1^{s_{11}}, \ldots,p_1^{s_{11}}) \cdot U_1 \cdot \diag(1,p_1,\ldots,p_1).
  \end{equation*}
  Then $F = AF_2$. Clearly, $p_1^{s_{11}}\mid d_1(A)$. Since $d_1(F) = p_1^{s_{11}} \cdot d_1(F_2)$ and $\J_1(F) = \R$, it follows from the Cauchy-Binet formula that
  \begin{equation*}\label{427-lem-equ-05}
  \J_1(A) = \J_1(F_2) = \R ~ \text{and} ~ d_1(A) =  p_1^{s_{11}}.
  \end{equation*}
  Based on Lemma \ref{424-lemma-1}, we have
  \begin{equation*}\label{427-lem-equ-06}
   A\sim_{\R} S_A = \diag(p_1^{s_{11}}, p_1^{s_{11}+1}, \ldots,p_1^{s_{11}+1}),
  \end{equation*}
  i.e., there are $V_1,V_2\in \GL_l(\R)$ such that $A = V_1S_AV_2$. Set $F_3 = V_2F_2$. Then
  \begin{equation*}\label{427-lem-equ-07}
   F_3\sim_{\R} F_2 ~ \text{and} ~ F \sim_{\R} S_A F_3 = \diag(p_1^{s_{11}}, p_1^{s_{11}+1}, \ldots,p_1^{s_{11}+1}) \cdot F_3.
  \end{equation*}
  Repeating the above computational process a total of $s_{21}-s_{11}$ times, we obtain
  \begin{equation*}\label{427-lem-equ-08}
   F\sim_{\R} \diag(p_1^{s_{11}}, p_1^{s_{21}}, \ldots,p_1^{s_{21}}) \cdot F_{N_{21}},
  \end{equation*}
  where $F_{N_{21}}\in \M_{l\times l}(\R)$. For $j$ ranging from $2$ to $l-1$, we successively compare the order relations between $s_{j1}$ and $s_{j+1,1}$ by the same method above, until we have
  \begin{equation*}\label{427-lem-equ-09}
   F\sim_{\R} \diag(p_1^{s_{11}}, p_1^{s_{21}}, p_1^{s_{31}}, \ldots,p_1^{s_{l1}}) \cdot F_{N_{l1}} = S_{p_1} \cdot F_{N_{l1}},
  \end{equation*}
  where $F_{N_{l1}}\in \M_{l\times l}(\R)$. Let the Smith normal form of $F$ w.r.t. $p_2$ be
  \begin{equation*}\label{427-lem-equ-010}
   S_{p_2} = \diag(p_2^{s_{12}}, p_2^{s_{22}}, \ldots,p_2^{s_{l2}}),
  \end{equation*}
  where $s_{12},s_{22},\ldots,s_{l2}$ are integers satisfying $0\leq s_{12} \leq s_{22} \leq \cdots \leq s_{l2}$ and $t_2 = \sum_{j=1}^{l}s_{j2}$. Since $\gcd(\det(S_{p_1}), \det(F_{N_{l1}})) = 1$, by Lemma \ref{Lu-2026-lemma-LAA} we have that the Smith normal form of $F_{N_{l1}}$ w.r.t. $p_2$ is $S_{p_2}$, and $\J_i(F_{N_{l1}}) = \R$ for $i=1,\ldots,l$. Following the same argument as above, we conclude that
  \begin{equation*}\label{427-lem-equ-011}
   F_{N_{l1}}\sim_{\R} S_{p_2}\cdot F_{N_{l2}},
  \end{equation*}
  where $F_{N_{l2}}\in \M_{l\times l}(\R)$. Let the Smith normal form of $F$ w.r.t. $p_k$ be
  \begin{equation*}\label{427-lem-equ-012}
   S_{p_k} = \diag(p_k^{s_{1k}}, p_k^{s_{2k}}, \ldots,p_k^{s_{lk}}),
  \end{equation*}
  where $k=3,\ldots,r$, $s_{1k},s_{2k},\ldots,s_{lk}$ are integers with $0\leq s_{1k} \leq s_{2k} \leq \cdots \leq s_{lk}$ and $t_k = \sum_{j=1}^{l}s_{jk}$. By repeating the above reasoning for all integers $k$ with $3\leq k \leq r$, we deduce that
  \begin{equation*}\label{427-lem-equ-013}
   F\sim_{\R} S_{p_1}Q_1 S_{p_2} Q_2 \cdots S_{p_r} Q_r F_{N_{ll}},
  \end{equation*}
  where $Q_1,\ldots,Q_r\in \GL_l(\R)$ and $F_{N_{ll}}\in \M_{l\times l}(\R)$. Then there are $Q_0,Q_{r+1}\in \GL_l(\R)$ such that
  \begin{equation*}\label{427-lem-equ-014}
   F = Q_0S_{p_1}Q_1S_{p_2}Q_2 \cdots S_{p_r} Q_r F_{N_{ll}} Q_{r+1}.
  \end{equation*}
  Assume that $\delta = \det(Q_0Q_1\cdots Q_{r-1})$. Then $\delta\in \K \setminus \{0\}$. Let $M = \diag(1,\ldots,1, \frac{\alpha}{\delta})$. Then $M\in \GL_l(\R)$. Set
  \begin{equation*}\label{427-lem-equ-015}
   G = Q_0S_{p_1}Q_1S_{p_2}Q_2 \cdots S_{p_r}M ~ \text{and} ~ H = M^{-1}Q_r F_{N_{ll}} Q_{r+1}.
  \end{equation*}
  Then $F = GH$ and $\det(G) = g$. \rulex
 \end{proof}

 We next present two analogous lemmas, which differ from Lemmas \ref{419-lemma-1} and \ref{423-lemma-1} only by replacing the irreducible polynomial $p$ with $\varphi_r$ and the univariate polynomial $g \in \K[x_1]$ with a power of $\varphi_r$. Their proofs are essentially identical and are omitted.

 \begin{lemma}\label{420-coro-1}
  Let $F\in \M_{l\times l}(\R)$ with $\varphi_r\mid \det(F)$, where $r$ is an integer with $2\leq r \leq n$. Let $\diag(\varphi_r^{s_{1r}},\ldots,\varphi_r^{s_{lr}})$ be the Smith normal form of $F$ w.r.t. $\varphi_r$, where $s_{1r},\ldots,s_{lr}$ are integers satisfying $0\leq s_{1r}\leq \cdots \leq s_{lr}$. Suppose there exist integers $k ~ (1\leq k \leq l-1)$ and $s ~ (s_{kr} \leq s < s_{k+1,r})$ such that
  \[ F \sim_{\R} \diag(\varphi_r^{s_{1r}},\ldots,\varphi_r^{s_{kr}},\varphi_r^s,\ldots,\varphi_r^s) \cdot G,\]
  where $G\in \M_{l\times l}(\R)$. If $\J_i(F) = \R$ for $i=1,\ldots,l$, then
  \begin{enumerate}
    \item $d_i(G) = \frac{d_i(F)}{\varphi_r^{s_{1r}+\cdots+s_{ir}}}$ and $\J_i(G) = \R$ for $i=1,\ldots,k$;

    \item $\rank(G(x_1,\ldots,x_{r-1},f_r,x_{r+1},\ldots,x_n)) = k$.
  \end{enumerate}
 \end{lemma}

 \begin{lemma}\label{425-lemma-3}
  Let $F\in \M_{l\times l}(\R)$ with $\det(F) = f\varphi_r^{t_r}$, where $r$ is an integer with $2\leq r \leq n$, $f\in \R$ satisfies $\gcd(f,\varphi_r)=1$, and $t_r$ is a positive integer. If $\J_i(F) = \R$ for $i=1,\ldots,l$, then there exist $F_1,F_2\in \M_{l\times l}(\R)$ such that $F=F_1F_2$ and $\det(F_2) = \varphi_r^{t_r}$.
 \end{lemma}

 The proof of Lemma \ref{425-lemma-3} differs slightly from that of Lemma \ref{423-lemma-1}. We first take the transpose of $F$ and set $A = F^{\rm T}$. We then factorize $A$ following the same line of reasoning as in Lemma \ref{423-lemma-1} to obtain $A=BC$ with $\det(B) = \varphi_r^{t_r}$. The desired result follows immediately by setting $F_1 = C^{\rm T}$ and $F_2 = B^{\rm T}$. This argument is simpler than that of Lemma \ref{423-lemma-1}, as it only involves extracting a power of $\varphi_r$, whereas Lemma \ref{423-lemma-1} requires factorizing $g\in \K[x_1]$ into distinct irreducible factors and successively extracting the powers of each factor.

 With the above matrix factorization results established, we now present the following lemma.

 \begin{lemma}\label{425-lemma-2}
  Let $A\in \M_{l\times l}(\R)$ be of full rank, and let $g_1,\ldots,g_l\in \K[x_1]$ satisfy $g_1\mid \cdots \mid g_l$. Suppose there exist an integer $k$ with $1\leq k \leq l-1$ and $U\in \GL_l(\R_2)$ such that
  \[A = \diag(g_1\varphi_2^{s_1},\ldots,g_k\varphi_2^{s_k},g_{k+1}\varphi_2^{s},
  \ldots,g_l\varphi_2^{s}) \cdot U \cdot \diag(\underbrace{1,\ldots,1}_k,\varphi_2,\ldots,\varphi_2),\]
  where $s_1,\ldots,s_k$ and $s$ are integers with $0\leq s_1\leq \cdots \leq s_k \leq s$. If $d_i(A) = g_1\cdots g_i \varphi_2^{s_1+\cdots +s_i}$ for $i=1,\ldots,k$, and $\J_i(A) = \R$ for $i=1,\ldots,l$, then
  \[A\sim_{\R} \diag(g_1\varphi_2^{s_1},\ldots,g_k\varphi_2^{s_k},g_{k+1}\varphi_2^{s+1},
  \ldots,g_l\varphi_2^{s+1}).\]
 \end{lemma}

 \begin{proof}
  Let $S = \diag(g_1\varphi_2^{s_1},\ldots,g_k\varphi_2^{s_k},
  g_{k+1}\varphi_2^{s+1},\ldots,g_l\varphi_2^{s+1})$. We first prove that $S$ is the Smith normal form of $A$. Put
  \[A_{01} = \diag(g_1,\ldots,g_l) ~ \text{and} ~ A_{02} = \diag(\varphi_2^{s_1},\ldots,\varphi_2^{s_k},\varphi_2^{s},
  \ldots,\varphi_2^{s}) \cdot U \cdot \diag(1,\ldots,1,\varphi_2,\ldots,\varphi_2).\]   
  Then $A = A_{01}A_{02}$. Since $\gcd(\det(A_{01}),\det(A_{02})) = 1$, Lemma \ref{Lu-2026-lemma-LAA} yields $d_i(A_{02}) = \varphi_2^{s_1+\cdots +s_i}$ and $\J_i(A_{02}) = \R$ for $i=1,\ldots,k$. It follows from Lemma \ref{502-lemma-1} that
  \[A_{02}\sim_{\R} \diag(\varphi_2^{s_1},\ldots,\varphi_2^{s_k},\varphi_2^{s+1},
  \ldots,\varphi_2^{s+1}).\]
  Using Lemma \ref{Lu-2026-lemma-LAA} again, $S$ is the Smith normal form of $A$.

  In the following, we prove $A\sim_{\R} S$ by considering the next three cases separately.

  \textbf{First Case:} $s_1=\cdots=s_k = s$.

  Let
  \[A_1 = \diag(g_1,\ldots,g_k,g_{k+1},
  \ldots,g_l) \cdot U \cdot \diag(1,\ldots,1,\varphi_2,\ldots,\varphi_2).\]
  Then $A = \varphi_2^{s} \cdot A_1$. Clearly, $\J_i(A_1) = \R$ for $i=1,\ldots,l$, and the Smith normal form of $A_1$ is
  \[S_{A_1} = \diag(g_1,\ldots,g_k,g_{k+1}\varphi_2,\ldots,g_l\varphi_2).\]
  Set $U = (u_{ij})_{l\times l}$, where $u_{ij}\in \R_2$ for $1\leq i,j\leq l$. Then
  \[ A_1 = \begin{pmatrix}
    g_1u_{11} & \cdots & g_1u_{1k} & g_1\varphi_2u_{1,k+1} & \cdots & g_1\varphi_2u_{1l} \\
    g_2u_{21} & \cdots & g_2u_{2k} & g_2\varphi_2u_{2,k+1} & \cdots & g_2\varphi_2u_{2l} \\
     \vdots & \ddots & \vdots & \vdots & \ddots & \vdots \\
    g_lu_{l1} & \cdots & g_lu_{lk} & g_l\varphi_2u_{l,k+1} & \cdots & g_l\varphi_2u_{ll}
    \end{pmatrix}.\]
  Let $B$ be the $l\times k$ submatrix formed by the first $k$ columns of $A_1$. We assert that
  \begin{equation*}\label{427-lemma-equ-1}
   d_i(B) = g_1\cdots g_i ~ \text{and} ~ \J_i(B) = \R_2 ~ \text{for} ~ i=1,\ldots,k.
  \end{equation*}
  For any given integer $i_0$ with $1\leq i_0 \leq k$, assume that $\alpha_1^{(i_0)},\ldots,\alpha_{N_{i_0}}^{(i_0)}\in \R_2$ are all $i_0\times i_0$ minors of $B$. Since $d_{i_0}(A_1) = g_1\cdots g_{i_0}$ and $B$ is a submatrix of $A_1$,
  \[ d_{i_0}(A_1)\mid \alpha_j^{(i_0)} ~ \text{for} ~j=1,\ldots,N_{i_0}.\]
  Set $h_j = \frac{\alpha_j^{(i_0)}}{d_{i_0}(A_1)}$, we obtain $h_j\in \R_2$ for $j=1,\ldots,N_{i_0}$. Let $\mathcal{C}$ be the set formed by all $i_0\times i_0$ reduced minors of $A_1$. Then $\{h_1,\ldots,h_{N_{i_0}}\}$ is a part of $\mathcal{C}$. If $\langle h_1,\ldots,h_{N_{i_0}} \rangle_{\R_2} \neq \R_2$, then there exists $\vec{\omega} = (\omega_1,\omega_3,\ldots,\omega_n)\in \overline{\K}^{n-1}$ such that
  \begin{equation*}\label{427-lemma-equ-2}
   h_j(\vec{\omega}) = 0 ~ \text{for} ~ j=1,\ldots,N_{i_0}.
  \end{equation*}
  For any $h\in \mathcal{C} \setminus \{h_1,\ldots,h_{N_{i_0}}\}$, it follows readily from the structure of $A_1$ and $\gcd(d_{i_0}(A_1),\varphi_2)=1$ that $\varphi_2\mid h$. Let $\omega_2 = f_2(\omega_1)$. Then
  \begin{equation}\label{427-lemma-equ-3}
   (\omega_1,\omega_2,\omega_3,\ldots,\omega_n)\in \mathds{V}\left(\langle \mathcal{C}\rangle_{\R}\right).
  \end{equation}
  As $\langle \mathcal{C}\rangle_{\R} = \J_{i_0}(A_1)$, Equation \eqref{427-lemma-equ-3} contradicts the fact that $\J_{i_0}(A_1) = \R$. Thus, $\langle h_1,\ldots,h_{N_{i_0}} \rangle_{\R_2}$ $= \R_2$. It follows that $d_{i_0}(B) = g_1\cdots g_{i_0}$ and $\J_{i_0}(B) = \R_2$. According to Lemma \ref{415-lemma-1}, we obtain
  \[B\sim_{\R_2} \begin{pmatrix}
                  g_1 &       &     \\
                      &\ddots &     \\
                      &       & g_k \\
                      &       &      \\
                      &       &
                \end{pmatrix} \triangleq S_B.\]
  Applying suitable unimodular transformations over $\R_2$ to the rows and the first $k$ columns of $A_1$, we conclude that
  \[A_1\sim_{\R_2} \begin{pmatrix}
    g_1 &        &     & \varphi_2v_{1,k+1} & \cdots & \varphi_2v_{1l} \\
        & \ddots &     & \vdots             & \ddots & \vdots  \\
        &        & g_k & \varphi_2v_{k,k+1} & \cdots & \varphi_2v_{kl} \\
        &        &     & \varphi_2v_{k+1,k+1} & \cdots & \varphi_2v_{k+1,l} \\
        &        &     & \vdots & \ddots & \vdots \\
        &        &     & \varphi_2v_{l,k+1} & \cdots & \varphi_2v_{ll}
    \end{pmatrix} \triangleq A_2,\]
  where $v_{ij}\in \R_2$ for $1\leq i \leq l$ and $k+1\leq j \leq l$. For any given integer $i_0$ with $1 \leq i_0 \leq k$, let
  \[D_j^{(i_0)} = \begin{pmatrix}
          g_1 &        &         & \varphi_2v_{1j} \\
              & \ddots &         & \vdots    \\
              &        & g_{i_0-1} & \varphi_2v_{i_0-1,j} \\
              &        &         & \varphi_2v_{i_0,j}
        \end{pmatrix} ~ \text{for} ~ j=k+1,\ldots,l.\]
  Since $D_j^{(i_0)}$ is an $i_0\times i_0$ submatrix of $A_2$, $d_{i_0}(A_2)\mid \det(D_j^{(i_0)})$. It follows that $g_{i_0}\mid v_{i_0,j}$ for $j=k+1,\ldots,l$. Finitely many elementary column operations over $\R$ on $A_2$ yield
  \[A_2 \sim_{\R} \begin{pmatrix}
    g_1 &        &     &   &   &   \\
        & \ddots &     &   &   &   \\
        &        & g_k &   &   &   \\
        &        &     & \varphi_2v_{k+1,k+1} & \cdots & \varphi_2v_{k+1,l} \\
        &        &     & \vdots & \ddots & \vdots \\
        &        &     & \varphi_2v_{l,k+1} & \cdots & \varphi_2v_{ll}
    \end{pmatrix} \triangleq A_3.\]
  Set $V = (v_{ij})\in \M_{(l-k)\times (l-k)}(\R_2)$, where $k+1\leq i,j\leq l$. Since $\R_2 \subset \R$,
  \begin{equation}\label{427-lemma-equ-4}
   A_1 \sim_{\R} A_3 = \diag(g_1,\ldots,g_k,\varphi_2,\ldots,\varphi_2) \cdot \diag(\mathbf{I}_k, V).
  \end{equation}
  Since $d_{k+1}(A_3) = d_{k+1}(A_1) = g_1\cdots g_k g_{k+1}\varphi_2$, the $(k+1)\times (k+1)$ minors containing the first $k$ diagonal entries of $A_3$ yield $g_{k+1}\mid d_1(V)$. Then
  \begin{equation}\label{920-lemma-equ-4-2}
   A_1 \sim_{\R} \diag(g_1,\ldots,g_k,g_{k+1}\varphi_2,\ldots,g_{k+1}\varphi_2) \cdot \diag(\mathbf{I}_k, \frac{1}{g_{k+1}}V).
  \end{equation}
  Applying Lemma \ref{422-lemma-1}, we obtain
  \begin{equation*}\label{427-lemma-equ-5}
   d_i(V) = \frac{d_{k+i}(A_1)}{g_1\cdots g_k \varphi_2^i}=g_{k+1}\cdots g_{k+i} ~ \text{and} ~ \J_i(V) = \J_{k+i}(A_1)=\R ~ \text{for} ~ i = 1,\ldots, l-k.
  \end{equation*}
  Using Lemma \ref{415-lemma-1} again,
  \begin{equation}\label{427-lemma-equ-6}
   V\sim_{\R} \diag(g_{k+1},\ldots,g_l).
  \end{equation}
  Combining Equations \eqref{427-lemma-equ-4} and \eqref{427-lemma-equ-6}, we have  $A_1\sim_{\R} S_{A_1}$. Consequently, $A\sim_{\R} S$.

  \textbf{Second Case:} $s_k < s$.

  Let $U = (u_{ij})_{l\times l}$, where $u_{ij}\in \R_2$ for $1\leq i,j\leq l$. Then
  \[ A = \diag(g_1\varphi_2^{s_1},\ldots,g_k\varphi_2^{s_k},
         g_{k+1}\varphi_2^{s},\ldots,g_l\varphi_2^{s}) \cdot
    \begin{pmatrix}\begin{smallmatrix}
     u_{11} & \cdots & u_{1k} & \varphi_2 u_{1,k+1} & \cdots & \varphi_2 u_{1l} \\
     \vdots & \ddots & \vdots & \vdots & \ddots & \vdots \\
     u_{k1} & \cdots & u_{kk} & \varphi_2 u_{k,k+1} & \cdots & \varphi_2 u_{kl} \\
     u_{k+1,1} & \cdots & u_{k+1,k} & \varphi_2 u_{k+1,k+1} & \cdots & \varphi_2 u_{k+1,l} \\
     \vdots & \ddots & \vdots & \vdots & \ddots & \vdots \\
     u_{l1} & \cdots & u_{lk} & \varphi_2 u_{l,k+1} & \cdots & \varphi_2 u_{ll}
    \end{smallmatrix}\end{pmatrix}.\]
   Set
   \[ U_{kk} = \begin{pmatrix} u_{11} & \cdots & u_{1k} \\ \vdots & \ddots & \vdots \\ u_{k1} & \cdots & u_{kk} \end{pmatrix}.\]
   Let $\det(U_{kk}) = u \in \R_2$. We assert that $u$ is a nonzero constant in $\R_2$. If otherwise, there exists $\vec{\omega}=(\omega_1,\omega_3,\ldots,\omega_n)\in \overline{\K}^{n-1}$ such that $u(\vec{\omega}) = 0$. Since $d_k(A) = g_1\cdots g_k \varphi_2^{s_1+\cdots+s_k}$ and $s_k < s$, $\varphi_2$ divides all $k\times k$ reduced minors of $A$ with the exception of $u$. Let $\omega_2 = f_2(\omega_1)$. Then $(\omega_1,\omega_2,\omega_3,\ldots,\omega_n)\in \mathds{V}(\J_k(A))$. This contradicts the fact that $\J_k(A) = \R$. Therefore, $U_{kk}\in \GL_{k}(\R_2)$. Then there exists $V_{kk}\in \GL_{k}(\R_2)$ such that $U_{kk} V_{kk} = \mathbf{I}_k$. It follows that
   \[ A \cdot \diag(V_{kk}, \mathbf{I}_{l-k}) =  \begin{pmatrix}\begin{smallmatrix}
     g_1\varphi_2^{s_1} &   &   & g_1\varphi_2^{s_1+1} u_{1,k+1} & \cdots & g_1\varphi_2^{s_1+1} u_{1l} \\ & \ddots &  & \vdots & \ddots & \vdots \\
      &  & g_k\varphi_2^{s_k} & g_k\varphi_2^{s_k+1} u_{k,k+1} & \cdots & g_k\varphi_2^{s_k+1} u_{kl} \\
     g_{k+1}\varphi_2^{s}v_{k+1,1} & \cdots & g_{k+1}\varphi_2^{s}v_{k+1,k} & g_{k+1}\varphi_2^{s+1} u_{k+1,k+1} & \cdots & g_{k+1}\varphi_2^{s+1} u_{k+1,l} \\ \vdots & \ddots & \vdots & \vdots & \ddots & \vdots \\
     g_l\varphi_2^{s}v_{l1} & \cdots & g_l\varphi_2^{s}v_{lk} & g_l\varphi_2^{s+1} u_{l,k+1} & \cdots & g_l\varphi_2^{s+1} u_{ll}
    \end{smallmatrix}\end{pmatrix},\]
    where $v_{ij}\in \R_2$ for $k+1\leq i \leq l$ and $1\leq j \leq k$. Clearly, $\diag(V_{kk}, \mathbf{I}_{l-k})\in \GL_l(\R_2)$. By performing finitely many elementary row and column operations over $\R$ on $A \cdot \diag(V_{kk}, \mathbf{I}_{l-k})$, we obtain
  \begin{equation}\label{427-lemma-equ-7}
   A\sim_{\R} \begin{pmatrix}
     g_1\varphi_2^{s_1} &   &   &   &   &   \\ & \ddots &  &   &   &  \\
      &  & g_k\varphi_2^{s_k} &   &   &   \\
       &   &   & g_{k+1}\varphi_2^{s+1} v_{k+1,k+1} & \cdots & g_{k+1}\varphi_2^{s+1} v_{k+1,l} \\
         &   &   & \vdots & \ddots & \vdots \\
       &   &   & g_l\varphi_2^{s+1} v_{l,k+1} & \cdots & g_l\varphi_2^{s+1} v_{ll}
    \end{pmatrix} \triangleq A_1,
  \end{equation}
    where $v_{ij}\in \R_2$ for $k+1\leq i,j \leq l$. Let
    \[B = \begin{pmatrix}
         v_{k+1,k+1} & v_{k+1,k+2} & \cdots &  v_{k+1,l} \\
         \frac{g_{k+2}}{g_{k+1}}v_{k+2,k+1} & \frac{g_{k+2}}{g_{k+1}}v_{k+2,k+2} & \cdots & \frac{g_{k+2}}{g_{k+1}}v_{k+2,l} \\
         \vdots & \vdots & \ddots & \vdots \\
        \frac{g_{l}}{g_{k+1}} v_{l,k+1} & \frac{g_{l}}{g_{k+1}} v_{l,k+2} &\cdots & \frac{g_{l}}{g_{k+1}} v_{ll}
    \end{pmatrix}.\]
  Then it follows from Equation \eqref{427-lemma-equ-7} that
  \begin{equation}\label{427-lemma-equ-8}
   A_1 = \diag(g_1\varphi_2^{s_1},\ldots,g_k\varphi_2^{s_k}, g_{k+1}\varphi_2^{s+1},\ldots,g_{k+1}\varphi_2^{s+1}) \cdot \diag(\mathbf{I}_k, B).
  \end{equation}
  Since $A\sim_{\R} A_1$, by Lemma \ref{422-lemma-1} we obtain
   \[d_i(B) = \frac{g_{k+1}\cdots g_{k+i}}{g_{k+1}^i} ~ \text{and} ~ \J_i(B) = \J_{k+i}(A)=\R ~ \text{for} ~ i=1,\ldots,l-k.\]
  Using Lemma \ref{415-lemma-1} again, we have
  \begin{equation}\label{427-lemma-equ-9}
   B\sim_{\R} \diag(1,\frac{g_{k+2}}{g_{k+1}},\ldots,\frac{g_{l}}{g_{k+1}}).
  \end{equation}
  Combining Equations \eqref{427-lemma-equ-8} and \eqref{427-lemma-equ-9}, we have $A\sim_{\R} S$.

  \textbf{Third Case:} there exists an integer $\tau$ with $1\leq \tau \leq k-1$ such that
  \[s_1\leq \cdots \leq s_{\tau} < s_{\tau+1}  = \cdots = s_k = s.\]

  Let $U = (u_{ij})_{l\times l}$, where $u_{ij}\in \R_2$ for $1\leq i,j\leq l$. Then
  \[ A = \begin{pmatrix}
     g_1\varphi_2^{s_1}u_{11} & \cdots & g_1\varphi_2^{s_1}u_{1k} & g_1\varphi_2^{s_1+1}u_{1,k+1} & \cdots & g_1\varphi_2^{s_1+1}u_{1l} \\
     \vdots & \ddots & \vdots & \vdots & \ddots & \vdots \\
     g_{\tau}\varphi_2^{s_\tau}u_{\tau1} & \cdots & g_{\tau}\varphi_2^{s_\tau}u_{\tau k} & g_{\tau}\varphi_2^{s_\tau+1}u_{\tau,k+1} & \cdots & g_{\tau}\varphi_2^{s_\tau+1}u_{\tau l} \\
     g_{\tau+1}\varphi_2^{s}u_{\tau+1,1} & \cdots & g_{\tau+1}\varphi_2^{s}u_{\tau+1,k} & g_{\tau+1}\varphi_2^{s+1}u_{\tau+1,k+1} & \cdots & g_{\tau+1}\varphi_2^{s+1}u_{\tau+1,l} \\
     \vdots & \ddots & \vdots & \vdots & \ddots & \vdots \\
     g_l\varphi_2^{s}u_{l1} & \cdots & g_l\varphi_2^{s}u_{lk} & g_l\varphi_2^{s+1}u_{l,k+1} & \cdots & g_l\varphi_2^{s+1}u_{ll}
  \end{pmatrix}.\]
  Let $U_{\tau k}$ be the $\tau\times k$ submatrix formed by the first $\tau$ rows and the first $k$ columns of $U$. We assert that $U_{\tau k}$ is a ZLP matrix. If otherwise, there exists $\vec{\omega}=(\omega_1,\omega_3,\ldots,\omega_n)\in \mathds{V}(\I_{\tau}(U_{\tau k}))$. Let $\mathcal{C}_1$ be the set formed by all $\tau\times \tau$ minors of $U_{\tau k}$, and let $\mathcal{C}_2$ be the set formed by all $\tau\times \tau$ reduced minors of $A$. Then $\langle \mathcal{C}_1 \rangle_{\R_2} = \I_{\tau}(U_{\tau k})$ and $\langle \mathcal{C}_2 \rangle_{\R} = \J_{\tau}(A)$. Since $d_{\tau}(A) = g_1\cdots g_{\tau}\varphi_2^{s_1+\cdots+s_{\tau}}$, $\mathcal{C}_1$ is a part of $\mathcal{C}_2$. For any $h\in \mathcal{C}_2 \setminus \mathcal{C}_1$, it follows from $s_{\tau} < s$ that $\varphi_2\mid h$. Set $\omega_2 = f_2(\omega_1)$. Then $h(\omega_1,\omega_2,\omega_3,\ldots,\omega_n)=0$. This implies that \[(\omega_1,\omega_2,\omega_3,\ldots,\omega_n)\in\mathds{V}(\J_{\tau}(A)).\]
  This contradicts the fact that $\J_{\tau}(A) = \R$. Therefore, $U_{\tau k}$ is a ZLP matrix. According to the Quillen-Suslin theorem, there exists $V_{kk}\in \GL_{k}(\R_2)$ such that $U_{\tau k} V_{kk} = (\mathbf{I}_{\tau}, 0_{\tau\times (k-\tau)})$. By performing finitely many elementary row and column operations over $\R$ on $A\cdot \diag(V_{kk}, \mathbf{I}_{l-k})$, we obtain
  \[A \sim_{\R} \diag(g_1\varphi_2^{s_1},\ldots,g_{\tau}\varphi_2^{s_\tau}, g_{\tau+1}\varphi_2^{s},\ldots,g_{\tau+1}\varphi_2^{s})\cdot \diag(\mathbf{I}_{\tau},B) \triangleq A_1,\]
  where
  \[B = \begin{pmatrix} v_{\tau+1,\tau+1} & \cdots & v_{\tau+1,k} & \varphi_2v_{\tau+1,k+1} &  \cdots & \varphi_2v_{\tau+1,l} \\
  \frac{g_{\tau+2}}{g_{\tau+1}}v_{\tau+2,\tau+1} & \cdots & \frac{g_{\tau+2}}{g_{\tau+1}}v_{\tau+2,k} & \frac{g_{\tau+2}}{g_{\tau+1}}\varphi_2v_{\tau+2,k+1} &  \cdots & \frac{g_{\tau+2}}{g_{\tau+1}}\varphi_2v_{\tau+2,l} \\
  \vdots & \ddots & \vdots & \vdots & \ddots & \vdots \\
  \frac{g_{l}}{g_{\tau+1}}v_{l,\tau+1} & \cdots & \frac{g_{l}}{g_{\tau+1}}v_{l,k} & \frac{g_{l}}{g_{\tau+1}}\varphi_2v_{l,k+1} &  \cdots & \frac{g_{l}}{g_{\tau+1}}\varphi_2v_{ll}
  \end{pmatrix},\]
  $v_{ij}\in \R_2$ for $\tau+1\leq i,j \leq l$. By Lemma \ref{422-lemma-1}, we have $\J_i(B)=\R$ for $i=1,\ldots,l-\tau$, and
  \[d_i(B) = \begin{cases}
               \frac{g_{\tau+1}\cdots g_{\tau+i}}{g_{\tau+1}^i}, & i= 1, \ldots, k-\tau; \\
               \frac{g_{\tau+1}\cdots g_{\tau+i}\varphi_2^{i-(k-\tau)}}{g_{\tau+1}^i}, & i = k-\tau+1,\ldots,l-\tau.
             \end{cases}
  \]
  It follows that the Smith normal form of $B$ is
  \begin{equation*}\label{427-lemma-equ-10}
   S_B = \diag(1, \frac{g_{\tau+2}}{g_{\tau+1}},\ldots,\frac{g_{k}}{g_{\tau+1}} ,\frac{g_{k+1}}{g_{\tau+1}}\varphi_2, \ldots, \frac{g_{l}}{g_{\tau+1}} \varphi_2).
  \end{equation*}
  Let $V_{l-\tau}= (v_{ij})\in \M_{(l-\tau)\times (l-\tau)}(\R_2)$, where $\tau+1\leq i,j \leq l$. Then
  \begin{equation*}\label{427-lemma-equ-11}
   B = \diag(1, \frac{g_{\tau+2}}{g_{\tau+1}},\ldots,\frac{g_{k}}{g_{\tau+1}} ,\frac{g_{k+1}}{g_{\tau+1}}, \ldots, \frac{g_{l}}{g_{\tau+1}})\cdot V_{l-\tau} \cdot \diag(\underbrace{1,\ldots,1}_{k-\tau},\varphi_2,\ldots,\varphi_2).
  \end{equation*}
  It is easy to verify that $V_{l-\tau}\in \GL_{l-\tau}(\R_2)$. Adopting the proof technique from the \textbf{First Case}, we obtain $B\sim_{\R} S_B$. Therefore, $A\sim_{\R} S$. \rulex
 \end{proof}

 Before turning to the solution of Problem \ref{main-problem}, we state one further technical lemma that is essential to our main arguments.

 \begin{lemma}\label{424-coro-1}
  Let $F\in \M_{l\times l}(\R)$ with $\det(F) = \varphi_r^t$, and let
  \[S_F = \diag(1,\ldots,1,\varphi_r^{s_{k+1}},\ldots,\varphi_r^{s_l})\]
  be the Smith normal form of $F$, where $k,r$ and $s_{k+1},\ldots,s_l,t$ are integers satisfying $1\leq k \leq l-1$, $2\leq r \leq n$, $1 \leq s_{k+1} \leq \cdots \leq s_l$ and $t = \sum_{j=k+1}^{l}s_j$. Suppose there exists $F_1\in \M_{l\times l}(\R)$ such that
  \[F\sim_{\R} \diag(\underbrace{1,\ldots,1}_k,\varphi_r,\ldots,\varphi_r) \cdot F_1.\]
  If $\J_i(F) = \R$ for $i=1,\ldots,l$, then
  \[F_1\sim_{\R} \diag(1,\ldots,1,\varphi_r^{s_{k+1}-1},\ldots,\varphi_r^{s_l-1}).\]
 \end{lemma}
 
 \begin{proof}
  Let $\Lambda = \diag(1,\ldots,1,\varphi_r,\ldots,\varphi_r)$, and $T = \diag(1,\ldots,1,\varphi_r^{s_{k+1}-1},\ldots,\varphi_r^{s_l-1})$. Then $S_F = \Lambda T$. Since $\J_i(F) = \R$ for $i=1,\ldots,l$, by Lemma \ref{425-lemma-1} we have $F\sim_{\R} S_F$. As $F\sim_{\R} \Lambda F_1$, there exist $U,V\in \GL_l(\R)$ such that $\Lambda F_1 = U S_F V$. Since $V \in \GL_l(\R)$, we obtain $U S_F = \Lambda F_1 V^{-1} \in \M_{l\times l}(\R)$. As $\Lambda = \diag(1,\ldots,1,\varphi_r,\ldots,\varphi_r)$, the last $l-k$ rows of $U S_F$ are divisible by $\varphi_r$. Set
  \[ U = \begin{pmatrix} U_{11} & U_{12} \\ U_{21} & U_{22} \end{pmatrix},\]
  where $U_{11}\in \M_{k\times k}(\R)$, $U_{12}\in \M_{k\times (l-k)}(\R)$, $U_{21}\in \M_{(l-k)\times k}(\R)$ and $U_{22}\in \M_{(l-k)\times (l-k)}(\R)$. Then $\varphi_r\mid d_1(U_{21})$. Combining the equations $S_F = \Lambda T$ and $\Lambda F_1 = U S_F V$, we have
  \[ F_1 = (\Lambda^{-1} U \Lambda) T V.\]
  Since $\det(\Lambda^{-1} U \Lambda)\in \K \setminus \{0\}$, and 
  \[\Lambda^{-1} U \Lambda = \begin{pmatrix} U_{11} & \varphi_r U_{12} \\ \frac{1}{\varphi_r}U_{21} & U_{22} \end{pmatrix} \in \M_{l\times l}(\R),\]
  we have $\Lambda^{-1} U \Lambda \in \GL_l(\R)$. This implies that $F_1\sim_{\R} T$. \rulex
 \end{proof}

 Drawing on the preceding Lemmas \ref{425-lemma-2} and \ref{424-coro-1}, which are central to the proof, we now propose Theorem \ref{425-Theorem-1}. This theorem addresses the special case of Problem \ref{main-problem}.

 \begin{theorem}\label{425-Theorem-1}
  Let $F\in \M_{l\times l}(\R)$ with $\det(F) = f_1 \varphi_2^t$, where $f_1\in \K[x_1]$ and $t$ is a nonnegative integer. Then $F$ is equivalent to its Smith normal form if and only if $\J_i(F) = \R$ for $i=1,\ldots,l$.
 \end{theorem}

 \begin{proof}
  Without loss of generality, assume that the Smith normal form of $F$ is
  \[S_F = \diag(g_1\varphi_2^{s_1},g_2\varphi_2^{s_2},\ldots,g_l\varphi_2^{s_l}),\]
  where $f_1 = \prod_{i=1}^{l}g_i$ and $g_1\mid \cdots \mid g_l$, $t=\sum_{i=1}^{l}s_i$ and $0\leq s_1\leq \cdots\leq s_l$.

  The necessity is obvious from Proposition \ref{lemma-reduced-2}. It suffices to prove the sufficiency. According to Lemma \ref{423-lemma-1}, there exist $G,H\in \M_{l\times l}(\R)$ such that
  \begin{equation*}\label{427-theorm-equ-1}
   F=GH ~ \text{and} ~ \det(G) = f_1.
  \end{equation*}
  Since $\gcd(\det(G),\det(H))=1$, it follows from Lemma \ref{Lu-2026-lemma-LAA} that
  \begin{equation*}\label{427-theorm-equ-2}
   d_i(F) = d_i(G)\cdot d_i(H) ~ \text{and} ~ \J_i(G) = \J_i(H) = \R ~ \text{for} ~ i=1,\ldots,l.
  \end{equation*}
  Based on Lemmas \ref{415-lemma-1} and \ref{425-lemma-1}, we have
  \begin{equation*}\label{427-theorm-equ-3}
   G\sim_{\R} S_G = \diag(g_1,g_2,\ldots,g_l) ~ \text{and} ~ H\sim_{\R} S_H = \diag(\varphi_2^{s_1},\varphi_2^{s_2},\ldots,\varphi_2^{s_l}).
  \end{equation*}
  Since $G\sim_{\R} S_G$, there are $U_1,V_1\in \GL_l(\R)$ such that $G = U_1S_GV_1$. Set $H_1 = V_1H$. Then $H_1\sim_{\R} H$. Let $H_2\in \M_{l\times l}(\R)$ satisfy $H_1 = \varphi_2^{s_1}\cdot H_2$. Clearly,
  \begin{equation*}\label{427-theorm-equ-4}
   F\sim_{\R} \diag(g_1\varphi_2^{s_1},g_2\varphi_2^{s_1},\ldots,g_l\varphi_2^{s_1})\cdot H_2 ~ \text{and} ~ H_2\sim_{\R} S_{H_2} = \diag(1,\varphi_2^{s_2-s_1},\ldots,\varphi_2^{s_l-s_1}).
  \end{equation*}
  If $s_1=s_2$, then we consider the order relation between $s_2$ and $s_3$. Otherwise, it follows from $H_2\sim_{\R} S_{H_2}$ that $\J_1(H_2) = \R$ and $\rank(H_2(x_1,f_2,x_3,\ldots,x_n))=1$. By Lemma \ref{424-lemma-2}, there exist $U_2\in \GL_l(\R_2)$ and $H_3\in \M_{l\times l}(\R)$ such that
  \begin{equation*}\label{427-theorm-equ-5}
   H_2 = U_2\cdot \diag(1,\varphi_2,\ldots,\varphi_2) \cdot H_3.
  \end{equation*}
  Using Lemma \ref{424-coro-1}, we have
  \begin{equation}\label{427-theorm-equ-6}
   H_3\sim_{\R} S_{H_3} = \diag(1,\varphi_2^{s_2-s_1-1},\ldots,\varphi_2^{s_l-s_1-1}).
  \end{equation}
  Let
  \begin{equation*}\label{427-theorm-equ-7}
   A = \diag(g_1\varphi_2^{s_1},g_2\varphi_2^{s_1},\ldots,
  g_l\varphi_2^{s_1})\cdot U_2 \cdot \diag(1,\varphi_2,\ldots,\varphi_2).
  \end{equation*}
  Then $F\sim_{\R} A H_3$. According to the Cauchy-Binet formula, we have $g_1\varphi_2^{s_1} \mid d_1(A) \mid d_1(F)$, and
  \[g_1\cdots g_i \varphi_2^{i \cdot s_1 + (i-1)} \mid d_i(A) ~ \text{and} ~ d_i(A) \cdot \varphi_2^{s_2+\cdots+s_i - (i-1)(s_1+1)} \mid d_i(F) ~ \text{for} ~ i=2,\ldots,l.\]
  This implies that the Smith normal form of $A$ is
  \begin{equation}\label{427-theorm-equ-8}
   S_A = \diag(g_1\varphi_2^{s_1},g_2\varphi_2^{s_1+1},\ldots,
  g_l\varphi_2^{s_1+1}).
  \end{equation}
  It follows from Equations \eqref{427-theorm-equ-6} and \eqref{427-theorm-equ-8} that $d_i(F) = d_i(A) \cdot d_i(H_3)$ for $i=1,\ldots,l$. Using the Cauchy-Binet formula again, $\J_i(A)=\R$ for $i=1,\ldots,l$. By Lemma \ref{425-lemma-2}, $A\sim_{\R} S_A$. Then there exist $U_3,V_3\in \GL_l(\R)$ such that $A=U_3S_AV_3$. Set $H_4 = V_3H_3$. Then
  \begin{equation*}\label{427-theorm-equ-9}
   F\sim_{\R} S_A \cdot H_4 ~ \text{and} ~ H_4 \sim_{\R} S_{H_4} = \diag(1,\varphi_2^{s_2-s_1-1},\ldots,\varphi_2^{s_l-s_1-1}).
  \end{equation*}
  By performing the above computational process a total of $s_2-s_1$ times, we deduce that
  \begin{equation*}\label{427-theorm-equ-10}
   F\sim_{\R} \diag(g_1\varphi_2^{s_1},g_2\varphi_2^{s_2},g_3\varphi_2^{s_2},
    \ldots,g_l\varphi_2^{s_2})\cdot F_{N_{s1}},
  \end{equation*}
  where $ F_{N_{s1}}\in \M_{l\times l}(\R)$ satisfies
  \begin{equation*}\label{427-theorm-equ-11}
   F_{N_{s1}} \sim_{\R} S_{F_{N_{s1}}} = \diag(1,1,\varphi_2^{s_3-s_2},\ldots,\varphi_2^{s_l-s_2}).
  \end{equation*}
  Applying the same argument as above, we successively compare the order relations between $s_i$ and $s_{i+1}$ for each $i=2,\ldots,l-1$, and conclude that
  \begin{equation*}\label{427-theorm-equ-12}
   F \sim_{\R} S_F = \diag(g_1\varphi_2^{s_1},g_2\varphi_2^{s_2},\ldots,g_l\varphi_2^{s_l}).
  \end{equation*} \rulex
 \end{proof}

 Having laid all the necessary groundwork, we now present the main result of this paper, which resolves Problem \ref{main-problem} in full generality.

 \begin{theorem}\label{425-Theorem-2}
  Let $F\in \M_{l\times l}(\R)$ with $\det(F) = f_1 \varphi_2^{t_2} \varphi_3^{t_3} \cdots \varphi_n^{t_n}$, where $f_1\in \K[x_1]$ and $t_2,\ldots,t_n$ are nonnegative integers. Then $F$ is equivalent to its Smith normal form if and only if $\J_i(F) = \R$ for $i=1,\ldots,l$.
 \end{theorem}

 \begin{proof}
  The necessity is obvious from Proposition \ref{lemma-reduced-2}. It suffices to prove the sufficiency. We proceed by induction on $k$, where the determinant of $F$ has the form $\det(F) = f_1 \varphi_2^{t_2} \varphi_3^{t_3} \cdots \varphi_k^{t_k}$.

  The statement is true for $k=2$ by Theorem \ref{425-Theorem-1}. Assume that the statement holds for all $k < r$, where $r$ is an integer with $3 \leq r \leq n$. For $k = r$, we first make the following claim.

  \vskip 6pt

  {\small{\em \textbf{Claim.} Let $A\in \M_{l\times l}(\R)$, and let $h_1,\ldots,h_l\in \K[x_1,\ldots,x_{r-1}]$ satisfy $\prod_{i=1}^{l}h_i = f_1 \varphi_2^{t_2} \varphi_3^{t_3} \cdots \varphi_{r-1}^{t_{r-1}}$ and $h_1\mid \cdots \mid h_l$. Suppose there exist an integer $\theta$ with $1\leq \theta \leq l-1$ and $U_\theta\in \GL_l(\R_r)$ such that
  \[A = \diag(h_1\varphi_r^{s_1},\ldots,h_\theta\varphi_r^{s_\theta},h_{\theta+1}\varphi_r^{s},
  \ldots,h_l\varphi_r^{s}) \cdot U_\theta \cdot \diag(\underbrace{1,\ldots,1}_\theta,\varphi_r,\ldots,\varphi_r),\]
  where $s_1,\ldots,s_\theta$ and $s$ are integers satisfying $0\leq s_1\leq \cdots \leq s_\theta \leq s$. If $d_i(A) = h_1\cdots h_i \varphi_r^{s_1+\cdots+s_i}$ for $i=1,\ldots,\theta$, and $\J_i(A) = \R$ for $i=1,\ldots,l$, then
  \[A\sim_{\R} \diag(h_1\varphi_r^{s_1},\ldots,h_\theta\varphi_r^{s_\theta},h_{\theta+1}\varphi_r^{s+1},
  \ldots,h_l\varphi_r^{s+1}).\]}}

  Using this claim, we then proceed to prove that $F$ satisfying $\det(F) = f_1 \varphi_2^{t_2} \varphi_3^{t_3} \cdots \varphi_r^{t_r}$ is equivalent to its Smith normal form.

  According to Lemma \ref{425-lemma-3}, there are $F_1,F_2\in \M_{l\times l}(\R)$ such that
  \begin{equation}\label{427-equ-1}
   F = F_1F_2 ~ \text{and} ~ \det(F_2) = \varphi_r^{t_r}.
  \end{equation}
  Since $\det(F) = \det(F_1)\cdot \det(F_2)$, $\det(F_1) = f_1 \varphi_2^{t_2} \varphi_3^{t_3} \cdots \varphi_{r-1}^{t_{r-1}}$. Clearly, $\gcd(\det(F_1),\det(F_2))=1$. It follows from Lemma \ref{Lu-2026-lemma-LAA} that
  \[d_i(F) = d_i(F_1)\cdot d_i(F_2) ~ \text{and} ~ \J_i(F_1) = \J_i(F_2) = \R ~ \text{for} ~ i=1,\ldots,l.\]
  Let
  \[S_{F_1} = \diag(h_1,\ldots,h_l)\]
  be the Smith normal form of $F_1$, where $h_1,\ldots,h_l\in \K[x_1,\ldots,x_{r-1}]$ satisfy $h_1\mid\cdots\mid h_l$ and $\det(F_1) = h_1\cdots h_l$. By the induction hypothesis,
  \begin{equation*}\label{427-equ-2}
   F_1\sim_{\R} S_{F_1}.
  \end{equation*}
  Then there are $U_1,V_1\in \GL_l(\R)$ such that
  \begin{equation}\label{427-equ-3}
   F_1 = U_1S_{F_1}V_1.
  \end{equation}
  Assume that
  \[S_{F_2} = \diag(\varphi_r^{t_{1r}},\ldots,\varphi_r^{t_{lr}})\]
  is the Smith normal form of $F_2$, where $t_{1r},\ldots,t_{lr}$ are integers satisfying $0\leq t_{1r} \leq \cdots \leq t_{lr}$ and $t_r = t_{1r}+\cdots+t_{lr}$. By Lemma \ref{425-lemma-1},
  \[F_2\sim_{\R} S_{F_2}.\]
  It follows from $d_i(F) = d_i(F_1)\cdot d_i(F_2)$ for $i=1,\ldots,l$ that the Smith normal form of $F$ is
  \[S_F = \diag(h_1\varphi_r^{t_{1r}},h_2\varphi_r^{t_{2r}},
    \ldots,h_l\varphi_r^{t_{lr}}).\]
  Let $F_{21}\in \M_{l\times l}(\R)$ satisfy
  \begin{equation}\label{427-equ-4}
   V_1F_2=\varphi_r^{t_{1r}} \cdot F_{21}.
  \end{equation}
  Then $\J_i(F_{21}) = \R$ for $i=1,\ldots, l$, and the Smith normal form of $F_{21}$ is
  \[S_{F_{21}} = \diag(1,\varphi_r^{t_{2r}-t_{1r}},\ldots,\varphi_r^{t_{lr}-t_{1r}}).\]
  If $t_{1r} = t_{2r}$, then we consider the order relation between $t_{2r}$ and $t_{3r}$. Otherwise, it follows from $F_{21}\sim_{\R} S_{F_{21}}$ that $\rank(F_{21}(x_1,\ldots,x_{r-1},f_r,x_{r+1}, \ldots,x_n))=1$. Based on Lemma \ref{424-lemma-2}, there are $U_{21}\in \GL_l(\R_r)$ and $F_{22}\in \M_{l\times l}(\R)$ such that
  \begin{equation}\label{427-equ-5}
   F_{21} = U_{21} \cdot \diag(1,\varphi_r,\ldots,\varphi_r) \cdot F_{22}.
  \end{equation}
  Using Lemma \ref{424-coro-1}, we have
  \begin{equation}\label{427-equ-6}
   F_{22}\sim_{\R} \diag(1,\varphi_r^{t_{2r}-t_{1r}-1},\ldots,\varphi_r^{t_{lr}-t_{1r}-1}) \triangleq S_{F_{22}}.
  \end{equation}
  Let
  \begin{equation}\label{427-equ-7}
   A = \diag(h_1\varphi_r^{t_{1r}},h_2\varphi_r^{t_{1r}},\ldots,
  h_l\varphi_r^{t_{1r}})\cdot U_{21} \cdot \diag(1,\varphi_r,\ldots,\varphi_r).
  \end{equation}
  Combining Equations \eqref{427-equ-1}--\eqref{427-equ-5}, we obtain
  \begin{equation*}\label{427-equ-8}
   F\sim_{\R} AF_{22}.
  \end{equation*}
  By the same argument used to derive Equation \eqref{427-theorm-equ-8} in Theorem \ref{425-Theorem-1}, the Smith normal form of $A$ is
  \begin{equation}\label{427-equ-9}
   S_A = \diag(h_1\varphi_r^{t_{1r}},h_2\varphi_r^{t_{1r}+1},\ldots,
  h_l\varphi_r^{t_{1r}+1}).
  \end{equation}
  It follows from Equations \eqref{427-equ-6} and \eqref{427-equ-9} that $d_i(F) = d_i(A)\cdot d_i(F_{22})$ for $i=1,\ldots,l$. According to the Cauchy-Binet formula, $\J_i(A)=\R$ for $i=1,\ldots,l$. By the Claim, $A\sim_{\R} S_A$. Then there are $U_3,V_3\in \GL_l(\R)$ such that $A = U_3S_AV_3$. Set $F_{23} = V_3F_{22}$. Then
  \begin{equation*}\label{427-equ-16}
   F\sim_{\R} S_A F_{23} ~ \text{and} ~ F_{23} \sim_{\R} S_{F_{22}}.
  \end{equation*}
  By performing the above computational process a total of $t_{2r}-t_{1r}$ times, we deduce that
  \begin{equation*}\label{427-equ-17}
   F\sim_{\R} \diag(h_1\varphi_r^{t_{1r}},h_2\varphi_r^{t_{2r}},h_3\varphi_r^{t_{2r}},
    \ldots,h_l\varphi_r^{t_{2r}})\cdot F_{N_{r1}},
  \end{equation*}
  where $ F_{N_{r1}}\in \M_{l\times l}(\R)$ satisfies
  \begin{equation*}\label{427-equ-18}
   F_{N_{r1}} \sim_{\R} S_{F_{N_{r1}}} = \diag(1,1,\varphi_r^{t_{3r}-t_{2r}},\ldots,\varphi_r^{t_{lr}-t_{2r}}).
  \end{equation*}
  Applying the same argument as above, we successively compare the order relations between $t_{jr}$ and $t_{(j+1)r}$ for each $j=2,\ldots,l-1$, and conclude that
  \begin{equation*}\label{427-equ-19}
   F \sim_{\R} S_{F} = \diag(h_1\varphi_r^{t_{1r}},h_2\varphi_r^{t_{2r}},
    \ldots,h_l\varphi_r^{t_{lr}}).
  \end{equation*} \rulex
 \end{proof}

 \begin{remark}
  The proofs of Theorems \ref{425-Theorem-1} and \ref{425-Theorem-2} proceed along largely parallel lines. The principal difference is that Theorem \ref{425-Theorem-1} depends crucially on Lemma \ref{425-lemma-2}, whereas the proof of Theorem \ref{425-Theorem-2} proceeds by mathematical induction in conjunction with an auxiliary claim. While this claim appears nearly identical to Lemma \ref{425-lemma-2} at first glance, its proof is inextricably tied to the inductive step and cannot be established as a standalone result. For the sake of completeness, the complete detailed proof is deferred to Appendix.
 \end{remark}

\section{Generalizations}\label{sec_general}

 Having established the main theorem (Theorem \ref{425-Theorem-2}) for full-rank square matrices over $\R$ in Section \ref{sec_main-results}, we now extend our results to non-square and rank-deficient matrices using the Quillen-Suslin theorem and the Lin-Bose lemma. We also consider the setting of automorphisms, thereby extending our analysis to a much wider class of matrices.

 \begin{theorem}\label{425-Theorem-3}
  Let $F\in \M_{l\times m}(\R)$ with rank $\gamma$, and $d_\gamma(F) = f_1 \varphi_2^{t_2} \varphi_3^{t_3} \cdots \varphi_n^{t_n}$, where $1\leq \gamma\leq \min\{l,m\}$, $f_1\in \K[x_1]$ and $t_2,\ldots,t_n$ are nonnegative integers. Then $F$ is equivalent to its Smith normal form if and only if $\J_i(F) = \R$ for $i=1,\ldots,\gamma$.
 \end{theorem}

 \begin{proof}
  The necessity is obvious from Proposition \ref{lemma-reduced-2}. It suffices to prove the sufficiency. Since $\J_\gamma(F) = \R$, by the Lin-Bose lemma there exist $F_1\in \M_{l\times \gamma}(\R)$ and $G_1\in \M_{\gamma\times m}(\R)$ such that
  \begin{equation}\label{430-equ-1}
   F = F_1G_1 ~ \text{with} ~ G_1 ~ \text{being ZLP}.
  \end{equation}
  According to the Quillen-Suslin theorem, there exists $V\in \GL_m(\R)$ such that
  \begin{equation}\label{430-equ-2}
   G_1 V = (\mathbf{I}_\gamma,0_{\gamma\times (m-\gamma)}).
  \end{equation}
  Combining Equations \eqref{430-equ-1} and \eqref{430-equ-2}, we obtain
  \begin{equation}\label{430-equ-3}
   F V = (F_1,0_{l\times (m-\gamma)}).
  \end{equation}
  Since $V\in \GL_m(\R)$, we have $F\sim_{\R}(F_1,0_{l\times (m-\gamma)})$. It follows from Proposition \ref{lemma-reduced-2} that
  \[\J_\gamma(F_1) = \J_\gamma(F) = \R.\]
  Using the Lin-Bose lemma again, there exist $F_2\in \M_{l\times \gamma}(\R)$ and $G_2\in \M_{\gamma\times \gamma}(\R)$ such that
  \begin{equation}\label{430-equ-4}
   F_1 = F_2G_2 ~ \text{with} ~ F_2 ~ \text{being ZRP}.
  \end{equation}
  Based on the Quillen-Suslin theorem, there exists $U\in \GL_l(\R)$ such that
  \begin{equation}\label{430-equ-5}
   UF_2 = \begin{pmatrix} \mathbf{I}_\gamma \\ 0_{(l-\gamma)\times \gamma}\end{pmatrix}.
  \end{equation}
  It follows from Equations \eqref{430-equ-3}--\eqref{430-equ-5} that
  \[UFV = \begin{pmatrix} G_2 & 0_{\gamma\times (m-\gamma)} \\ 0_{(l-\gamma)\times \gamma} & 0_{(l-\gamma)\times (m-\gamma)}\end{pmatrix}.\]
  Using Proposition \ref{lemma-reduced-2} again, we have
  \[d_i(G_2) = d_i(F) ~ \text{and} ~ \J_i(G_2) = \J_i(F) = \R ~\text{for}~ i=1,\ldots,\gamma.\]
  According to Theorem \ref{425-Theorem-2}, $G_2\sim_{\R} S_{G_2}$, where $S_{G_2}$ is the Smith normal form of $G_2$. Therefore, $F$ is equivalent to its Smith normal form. \rulex
 \end{proof}

 Set $f_1 = x_1^{t_1}$, and $f_j = 0$ for $j=2,\ldots,n$ in Theorem \ref{425-Theorem-3}. Then we can draw the following conclusion.

 \begin{corollary}\label{425-corollary-1}
  Let $F\in \M_{l\times m}(\R)$ with rank $\gamma$, and $d_\gamma(F) = x_1^{t_1} \cdots x_n^{t_n}$, where $1\leq \gamma \leq \min\{l,m\}$, and $t_1,\ldots,t_n$ are nonnegative integers. Then $F$ is equivalent to its Smith normal form if and only if $\J_i(F) = \R$ for $i=1,\ldots,\gamma$.
 \end{corollary}

 For matrices $F\in \M_{l\times m}(\R)$ with $d_\gamma(F) = x_1^{t_1}x_2^{t_2}$, Zeng et al.\supercite{Zeng2025Poly} studied the Smith normal form equivalence problem for matrices arising from $F$ under automorphisms of $\R$. Building on Corollary \ref{425-corollary-1}, we further explore matrix equivalence in the framework of automorphisms of $\R$.

 In the following, we assume that $\K$ is a field of characteristic zero. A $\K$-algebra automorphism of $\R$ is a bijective $\K$-linear map $\psi: ~ \R \rightarrow \R$ satisfying $\psi(fg) = \psi(f)\psi(g)$ for all $f,g\in \R$. The set of all $\K$-algebra automorphisms of $\R$ forms a group under composition, denoted by $\Aut_{\K}(\R)$. A key subclass of $\Aut_{\K}(\R)$ is formed by tame automorphisms, constructed from two elementary map families, whose definition follows.

 \begin{definition}[see \cite{Essen2000Polynomial}, Chapter 5, page 85]\label{def-tame}
  An automorphism $\psi\in \Aut_{\K}(\R)$ is called tame if it can be expressed as a finite composition of affine automorphisms and de Jonqui\`{e}res automorphisms, which are defined as follows:
  \begin{enumerate}
    \item affine automorphism: a map of the form
          \[\psi(x_1,\ldots,x_n) = \left(\sum_{j=1}^{n}a_{1j}x_j+c_1,\ldots,\sum_{j=1}^{n}a_{nj}x_j+c_n\right),\]
          where $A = (a_{ij})\in \GL_n(\K)$ and $c_1,\ldots,c_n\in \K$.

    \item de Jonqui\`{e}res automorphism: a map of the form
        \[\begin{cases}
            \psi(x_1)= a_1x_1+b_1, \\
            \psi(x_j) = a_jx_j+q_j(x_1,\ldots,x_{j-1}), ~ j=2,\ldots,n,
          \end{cases}\]
        where $a_1,\ldots,a_n\in \K\setminus \{0\}$, $b_1\in \K$, and $q_j\in \K[x_1,\ldots,x_{j-1}]$ for $j=2,\ldots,n$.
  \end{enumerate}
  The subgroup of $\Aut_{\K}(\R)$ consisting of all tame automorphisms is denoted by $\TA_{\K}(\R)$ and is called the tame automorphism group.
 \end{definition}

 With the notion of tame automorphisms now established, we state a more general version of Corollary \ref{425-corollary-1}.

 \begin{corollary}\label{425-corollary-2}
  Let $\psi \in \TA_{\K}(\R)$ and $g_1,\ldots,g_n\in \R$ satisfy $\psi(x_1,\ldots,x_n) = (g_1,\ldots,g_n)$. Let $F\in \M_{l\times m}(\R)$ with rank $\gamma$, and $d_\gamma(F) = g_1^{t_1} \cdots g_n^{t_n}$, where $1\leq \gamma \leq \min\{l,m\}$, and $t_1,\ldots,t_n$ are nonnegative integers. Then $F$ is equivalent to its Smith normal form if and only if $\J_i(F) = \R$ for $i=1,\ldots,\gamma$.
 \end{corollary}

 By the definition of de Jonqui\`{e}res automorphisms above, when $f_1 = (a_1x_1+b_1)^{t_1}$ in Theorem \ref{425-Theorem-3}, the tame automorphism $\psi$ in Corollary \ref{425-corollary-2} may be chosen as a de Jonqui\`{e}res automorphism, which realizes the setting of Theorem \ref{425-Theorem-3}. Otherwise, no automorphism of $\R$ can achieve such a transformation. Consequently, under automorphisms of $\R$, Corollary \ref{425-corollary-2} is a special case of Theorem \ref{425-Theorem-3}, corresponding to $f_1$ being a power of a linear form in $x_1$. Nevertheless, Corollary \ref{425-corollary-2} substantially extends the matrix form established in Theorem \ref{425-Theorem-3}, thereby allowing our framework to apply in a considerably broader setting.

 Next, we prove that the automorphism constructed in Lemma 3.2 of \cite{Zeng2025Poly} is tame, which shows that the equivalence classes studied in \citep{Zeng2025Poly} are special cases of our general framework.

 \begin{lemma}
  Let $\psi\in \Aut_{\K}(\K[x_1,x_2,x_3])$ be the automorphism defined by
  \[\begin{cases}
      \psi(x_1) = 2x_1+x_2, \\
      \psi(x_2) = 2x_2+x_3, \\
      \psi(x_3) = -4x_1^2+x_2^2+ax_1+bx_2+(c+2x_1+x_2)x_3,
    \end{cases}\]
  where $a,b,c\in \K$ satisfy $\Delta = a - 2b + 4c \neq 0$. Then $\psi\in \TA_{\K}(\K[x_1,x_2,x_3])$.
 \end{lemma}

 \begin{proof}
  Consider the automorphism $\psi_1$ defined by
  \[\psi_1(x_1,x_2,x_3) = (2x_1+x_2,2x_2+x_3,x_3).\]
  Clearly, the Jacobian matrix of $\psi_1$ is
  \[J = \begin{pmatrix} 2 & 1 & 0 \\
   0 & 2 & 1 \\ 0 & 0 & 1\end{pmatrix}.\]
  Since $\det(J) \neq 0$ in $\K$, $\psi_1$ is an affine automorphism. The inverse of $\psi_1$ is
  \[\psi_1^{-1}(y_1,y_2,y_3) = (\frac{2y_1-y_2+y_3}{4}, \frac{y_2-y_3}{2}, y_3).\]
  Set $\Phi = \psi_1^{-1} \circ \psi$. Then
  \[\Phi(y_1,y_2,y_3) = (y_1,y_2, \frac{\Delta}{4}y_3 - y_1^2+y_1y_2+\frac{a}{2}y_1+\frac{2b-a}{4}y_2).\]
  Since $\Delta \neq 0$, we define the affine automorphism $\psi_2$ by
  \[\psi_2(y_1,y_2,y_3) = (y_1,y_2,\frac{\Delta}{4}y_3),\]
  and the de Jonqui\`{e}res automorphism $\psi_3$ by
  \[\psi_3(u_1,u_2,u_3) = (u_1,u_2,u_3+q_3(u_1,u_2)),\]
  where $q_3(u_1,u_2) = -u_1^2+u_1u_2+\frac{a}{2}u_1+\frac{2b-a}{4}u_2$. Then
  \[(\psi_2 \circ \psi_3)(y_1,y_2,y_3) = (y_1,y_2, \frac{\Delta}{4}y_3 + q_3(y_1,y_2)) = \Phi(y_1,y_2,y_3).\]
  This implies that $\Phi = \psi_2 \circ \psi_3$. It follows that
  \[ \psi = \psi_1 \circ \psi_2 \circ \psi_3.\]
  Consequently, $\psi\in \TA_{\K}(\K[x_1,x_2,x_3])$. \rulex
 \end{proof}

 From Definition \ref{def-tame} and the preceding example, we have seen that all automorphisms considered thus far are tame. This naturally raises a fundamental question in polynomial automorphism theory: is every $\K$-algebra automorphism of $\R$ tame? The answer to this long-standing problem depends critically on the number of variables $n$. For $n=1$, i.e., $\R = \K[x_1]$, every automorphism is affine and hence tame. For $n=2$, i.e., $\R = \K[x_1,x_2]$, the Jung-van der Kulk theorem\supercite{jung1942ganze,vanderkulk1953polynomial} guarantees that all automorphisms are tame. For $n\geq 3$, the first explicit candidate for a non-tame (commonly called wild) automorphism of $\K[x_1,x_2,x_3]$ was constructed by Nagata\supercite{nagata1972autom}, who conjectured it cannot be expressed as a finite composition of affine and de Jonqui\`{e}res automorphisms. This Nagata conjecture was resolved affirmatively by Shestakov and Umirbaev\supercite{shestakov2003nagata,shestakov2004tame} in a landmark series of works.

 While tame automorphisms admit an explicit structure as finite compositions of affine and de Jonqui\`{e}res automorphisms, no analogous structural characterization exists for wild automorphisms. Consequently, although the Nagata automorphism and its generalizations provide explicit wild automorphisms in certain cases, for polynomial rings $\R = \K[x_1,\ldots,x_n]$ with $n\geq 3$ in general there is no systematic method for constructing wild automorphisms, and proving that a given automorphism is indeed wild remains notoriously difficult. This makes the study of matrix equivalence via explicit wild automorphism constructions a particularly challenging problem, and an important open direction for future research.

\section{Concluding Remarks}\label{sec_conclusions}

 This paper studies the classical problem of determining when a multivariate polynomial matrix is equivalent to its Smith normal form. This problem has been extensively investigated since the landmark conjecture of Frost and Storey\supercite{Frost1978}, which asserts that a polynomial matrix is equivalent to its Smith normal form if and only if the reduced minors of each order of the matrix generate the unit ideal. While this conjecture is known to be false in general, it has been proven to hold for several important classes of matrices. We consider a broad class of square matrices over $\K[x_1,x_2\ldots,x_n]$ and, using techniques from matrix theory and polynomial ideal theory, establish that Frost and Storey's conjecture remains valid for this class. Furthermore, by applying the Quillen-Suslin theorem and the Lin-Bose lemma, we extend our main result to non-square and rank-deficient matrices, and also discuss the matrix equivalence problem up to polynomial ring automorphisms. Our result unifies and extends several previous results in the literature.

 While this paper provides a complete theoretical resolution to the equivalence problem for the aforementioned class of matrices, the algorithmic aspects of this problem remain largely underdeveloped. The reduction of a multivariate polynomial matrix to its Smith normal form requires the construction of a sequence of unimodular transformations. Although Fabia\'{n}ska and Quadrat\supercite{Fabianska2007} have proposed an algorithm for computing such unimodular matrices based on the Quillen-Suslin theorem, its computational complexity is at least exponential, which severely limits its practical applicability. Consequently, the design of efficient algorithms for computing unimodular matrices over multivariate polynomial rings constitutes a significant open problem for future research. In addition, characterizing further classes of multivariate polynomial matrices for which Frost and Storey's conjecture holds remains a fundamental and challenging direction in this field.

\section*{Conflict of Interest}

 The authors declare no conflict of interest.

\section*{Appendix}\label{Appendix-A}

 We now present the detailed proof of the auxiliary claim employed in the proof of Theorem \ref{425-Theorem-2}. The overall strategy is similar to that of Lemma \ref{425-lemma-2}, but the argument here must be carried out in conjunction with the inductive hypothesis from the main theorem. The essential distinction is that in Lemma \ref{425-lemma-2} the polynomials $g_1,\ldots,g_l$ belong to the univariate polynomial ring $\K[x_1]$, which permits a direct application of Lemma \ref{415-lemma-1}. In the present claim, however, the polynomials $h_1,\ldots,h_l$ lie in the multivariate polynomial ring $\K[x_1,\ldots,x_{r-1}]$. Consequently, Lemma \ref{415-lemma-1} is no longer directly applicable; we instead rely on the Lin-Bose lemma and the Quillen-Suslin theorem to achieve the desired reduction.

 \vskip 6pt

 \noindent {\it Proof of Claim in Theorem \ref{425-Theorem-2}}

  Let $S = \diag(h_1\varphi_r^{s_1},\ldots,h_\theta\varphi_r^{s_\theta},h_{\theta+1}\varphi_r^{s+1},
  \ldots,h_l\varphi_r^{s+1})$. As in the opening of the proof of Lemma  \ref{425-lemma-2}, we conclude that $S$ is the Smith normal form of $A$. We shall prove $A\sim_{\R} S$ by considering the following three cases separately.

  \textbf{First Case:} $s_1=\cdots=s_\theta = s$.

  Let
  \[A_1 = \diag(h_1,\ldots,h_\theta,h_{\theta+1},
  \ldots,h_l) \cdot U_\theta \cdot \diag(1,\ldots,1,\varphi_r,\ldots,\varphi_r).\]
  Then $A = \varphi_r^{s} \cdot A_1$. Clearly, $\J_i(A_1) = \R$ for $i=1,\ldots,l$, and the Smith normal form of $A_1$ is
  \[S_{A_1} = \diag(h_1,\ldots,h_\theta,h_{\theta+1}\varphi_r,\ldots,h_l\varphi_r).\]
  Set $U_\theta = (u_{ij})_{l\times l}$, where $u_{ij}\in \R_r$ for $1\leq i,j\leq l$. Then
  \[ A_1 = \begin{pmatrix}
    h_1u_{11} & \cdots & h_1u_{1\theta} & h_1\varphi_ru_{1,\theta+1} & \cdots & h_1\varphi_ru_{1l} \\
    h_2u_{21} & \cdots & h_2u_{2\theta} & h_2\varphi_ru_{2,\theta+1} & \cdots & h_2\varphi_ru_{2l} \\
     \vdots & \ddots & \vdots & \vdots & \ddots & \vdots \\
    h_lu_{l1} & \cdots & h_lu_{l\theta} & h_l\varphi_ru_{l,\theta+1} & \cdots & h_l\varphi_ru_{ll}
    \end{pmatrix}.\]
  Let $B$ be the $l\times \theta$ submatrix formed by the first $\theta$ columns of $A_1$. We assert that
  \begin{equation*}\label{501-lemma-equ-711-1}
   d_i(B) = h_1\cdots h_i ~ \text{and} ~ \J_i(B) = \R_r ~ \text{for} ~ i=1,\ldots,\theta.
  \end{equation*}
  For any given integer $i_0$ with $1\leq i_0 \leq \theta$, assume that $\alpha_1^{(i_0)},\ldots,\alpha_{N_{i_0}}^{(i_0)}\in \R_r$ are all $i_0\times i_0$ minors of $B$. Since $d_{i_0}(A_1) = h_1\cdots h_{i_0}$ and $B$ is a submatrix of $A_1$,
  \[ d_{i_0}(A_1)\mid \alpha_j^{(i_0)} ~ \text{for} ~j=1,\ldots,N_{i_0}.\]
  Set $q_j = \frac{\alpha_j^{(i_0)}}{d_{i_0}(A_1)}$, we obtain $q_j\in \R_r$ for $j=1,\ldots,N_{i_0}$. Let $\mathcal{C}$ be the set formed by all $i_0\times i_0$ reduced minors of $A_1$. Then $\{q_1,\ldots,q_{N_{i_0}}\}$ is a part of $\mathcal{C}$. If $\langle q_1,\ldots,q_{N_{i_0}} \rangle_{\R_r} \neq \R_r$, then there exists $\vec{\omega} = (\omega_1,\ldots,\omega_{r-1}, \omega_{r+1},\ldots,\omega_n)\in \overline{\K}^{n-1}$ such that
  \begin{equation*}\label{501-lemma-equ-711-2}
   q_j(\vec{\omega}) = 0 ~ \text{for} ~ j=1,\ldots,N_{i_0}.
  \end{equation*}
  For any $q\in \mathcal{C} \setminus \{q_1,\ldots,q_{N_{i_0}}\}$, it follows readily from the structure of $A_1$ and $\gcd(d_{i_0}(A_1),\varphi_r)=1$ that $\varphi_r\mid q$. Let $\omega_r = f_r(\omega_1,\ldots,\omega_{r-1})$. Then
  \begin{equation}\label{501-lemma-equ-711-3}
   (\omega_1,\ldots,\omega_{r-1},\omega_{r},\omega_{r+1},\ldots,\omega_n)\in \mathds{V}\left(\langle \mathcal{C}\rangle_{\R}\right).
  \end{equation}
  As $\langle \mathcal{C}\rangle_{\R} = \J_{i_0}(A_1)$, Equation \eqref{501-lemma-equ-711-3} contradicts the fact that $\J_{i_0}(A_1) = \R$. Thus $\langle q_1,\ldots,q_{N_{i_0}} \rangle_{\R_r}$ $= \R_r$. It follows that $d_{i_0}(B) = h_1\cdots h_{i_0}$ and $\J_{i_0}(B) = \R_r$. According to Lemma \ref{Lin-Bose-theorem} (the Lin-Bose lemma), there exist $B_1\in \M_{l\times \theta}(\R_r)$ and $B_2\in \M_{\theta\times \theta}(\R_r)$ such that
  \begin{equation}\label{501-lemma-equ-711-3-2}
   B = B_1B_2 ~ \text{with} ~ B_1 ~ \text{being ZRP}.
  \end{equation}
  Based on the Quillen-Suslin theorem, there exists $B_3\in \GL_l(\R_r)$ such that \begin{equation}\label{501-lemma-equ-711-3-3}
   B_3B_1 = \begin{pmatrix} \mathbf{I}_\theta \\ 0_{(l-\theta)\times \theta}\end{pmatrix}.
  \end{equation}
  Combining Equations \eqref{501-lemma-equ-711-3-2} and \eqref{501-lemma-equ-711-3-3}, we have
  \[ B\sim_{\R_r} \begin{pmatrix} B_2 \\ 0_{(l-\theta)\times \theta}\end{pmatrix}.\]
  It follows from Proposition \ref{lemma-reduced-2} that
  \begin{equation*}\label{501-lemma-equ-711-3-4}
   d_i(B_2) = h_1\cdots h_i ~ \text{and} ~ \J_i(B_2) = \R_r ~ \text{for} ~ i=1,\ldots,\theta.
  \end{equation*}
  By the induction hypothesis, we obtain
  \[B_2\sim_{\R_r} \diag(h_1,\ldots,h_\theta).\]
  It follows that
  \[B\sim_{\R_r} \begin{pmatrix}
                  h_1 &       &     \\
                      &\ddots &     \\
                      &       & h_\theta \\
                      &       &      \\
                      &       &
                \end{pmatrix} \triangleq S_B.\]
  Applying suitable unimodular transformations over $\R_r$ to the rows and the first $\theta$ columns of $A_1$, we conclude that
  \[A_1\sim_{\R_r} \begin{pmatrix}
    h_1 &        &     & \varphi_rv_{1,\theta+1} & \cdots & \varphi_rv_{1l} \\
        & \ddots &     & \vdots             & \ddots & \vdots  \\
        &        & h_\theta & \varphi_rv_{\theta,\theta+1} & \cdots & \varphi_rv_{\theta l} \\
        &        &     & \varphi_rv_{\theta+1,\theta+1} & \cdots & \varphi_rv_{\theta+1,l} \\
        &        &     & \vdots & \ddots & \vdots \\
        &        &     & \varphi_rv_{l,\theta+1} & \cdots & \varphi_rv_{ll}
    \end{pmatrix} \triangleq A_2,\]
  where $v_{ij}\in \R_r$ for $1\leq i \leq l$ and $\theta+1\leq j \leq l$. For any given integer $i_0$ with $1 \leq i_0 \leq \theta$, let
  \[D_j^{(i_0)} = \begin{pmatrix}
          h_1 &        &         & \varphi_rv_{1j} \\
              & \ddots &         & \vdots    \\
              &        & h_{i_0-1} & \varphi_rv_{i_0-1,j} \\
              &        &         & \varphi_rv_{i_0,j}
        \end{pmatrix} ~ \text{for} ~ j=\theta+1,\ldots,l.\]
  Since $D_j^{(i_0)}$ is an $i_0\times i_0$ submatrix of $A_2$, $d_{i_0}(A_2)\mid \det(D_j^{(i_0)})$. It follows that $h_{i_0}\mid v_{i_0,j}$ for $j=\theta+1,\ldots,l$. Finitely many elementary column operations over $\R$ on $A_2$ yield
  \[A_2 \sim_{\R} \begin{pmatrix}
    h_1 &        &     &   &   &   \\
        & \ddots &     &   &   &   \\
        &        & h_\theta &   &   &   \\
        &        &     & \varphi_rv_{\theta+1,\theta+1} & \cdots & \varphi_rv_{\theta+1,l} \\
        &        &     & \vdots & \ddots & \vdots \\
        &        &     & \varphi_rv_{l,\theta+1} & \cdots & \varphi_rv_{ll}
    \end{pmatrix} \triangleq A_3.\]
  Set $V = (v_{ij})\in \M_{(l-\theta)\times (l-\theta)}(\R_r)$, where $\theta+1\leq i,j\leq l$. Since $\R_r \subset \R$,
  \begin{equation}\label{501-lemma-equ-711-4}
   A_1 \sim_{\R} A_3 = \diag(h_1,\ldots,h_\theta,\varphi_r,\ldots,\varphi_r) \cdot \diag(\mathbf{I}_\theta, V),
  \end{equation}
  By the same argument used to derive Equation \eqref{920-lemma-equ-4-2} in Lemma \ref{425-lemma-2}, we get
  \begin{equation*}\label{501-lemma-equ-711-5}
   d_i(V) = \frac{d_{\theta+i}(A_1)}{h_1\cdots h_\theta \varphi_r^i}=h_{\theta+1}\cdots h_{\theta+i} ~ \text{and} ~ \J_i(V) = \J_{\theta+i}(A_1)=\R ~ \text{for} ~ i = 1,\ldots, l-\theta.
  \end{equation*}
  By the induction hypothesis again,
  \begin{equation}\label{501-lemma-equ-711-6}
   V\sim_{\R_r} \diag(h_{\theta+1},\ldots,h_l).
  \end{equation}
  Combining Equations \eqref{501-lemma-equ-711-4} and \eqref{501-lemma-equ-711-6}, we have  $A_1\sim_{\R} S_{A_1}$. Consequently, $A\sim_{\R} S$.

  \textbf{Second Case:} $s_\theta < s$.

  Let $U_\theta = (u_{ij})_{l\times l}$, where $u_{ij}\in \R_r$ for $1\leq i,j\leq l$. Then
  \[ A = \diag(h_1\varphi_r^{s_1},\ldots,h_\theta\varphi_r^{s_\theta},
         h_{\theta+1}\varphi_r^{s},\ldots,h_l\varphi_r^{s}) \cdot
    \begin{pmatrix}\begin{smallmatrix}
     u_{11} & \cdots & u_{1\theta} & \varphi_r u_{1,\theta+1} & \cdots & \varphi_r u_{1l} \\
     \vdots & \ddots & \vdots & \vdots & \ddots & \vdots \\
     u_{\theta1} & \cdots & u_{\theta\theta} & \varphi_r u_{\theta,\theta+1} & \cdots & \varphi_r u_{\theta l} \\
     u_{\theta+1,1} & \cdots & u_{\theta+1,\theta} & \varphi_r u_{\theta+1,\theta+1} & \cdots & \varphi_r u_{\theta+1,l} \\
     \vdots & \ddots & \vdots & \vdots & \ddots & \vdots \\
     u_{l1} & \cdots & u_{l\theta} & \varphi_r u_{l,\theta+1} & \cdots & \varphi_r u_{ll}
    \end{smallmatrix}\end{pmatrix}.\]
   Set
   \[ U_{\theta\theta} = \begin{pmatrix} u_{11} & \cdots & u_{1\theta} \\ \vdots & \ddots & \vdots \\ u_{\theta1} & \cdots & u_{\theta\theta} \end{pmatrix}.\]
   Let $\det(U_{\theta\theta}) = u \in \R_r$. We assert that $u$ is a nonzero constant in $\R_r$. If otherwise, there exists $\vec{\omega}=(\omega_1,\ldots,\omega_{r-1}, \omega_{r+1},\ldots,\omega_n)\in \overline{\K}^{n-1}$ such that $u(\vec{\omega}) = 0$. Since $d_\theta(A) = h_1\cdots h_\theta \varphi_r^{s_1+\cdots+s_\theta}$ and $s_\theta < s$, $\varphi_r$ divides all $\theta\times \theta$ reduced minors of $A$ with the exception of $u$. Let $\omega_r = f_r(\omega_1,\ldots,\omega_{r-1})$. Then $(\omega_1,\ldots,\omega_{r-1},\omega_{r},\omega_{r+1},\ldots,\omega_n)\in \mathds{V}(\J_\theta(A))$. This contradicts the fact that $\J_\theta(A) = \R$. Therefore, $U_{\theta\theta}\in \GL_{\theta}(\R_r)$. Then there exists $V_{\theta\theta}\in \GL_{\theta}(\R_r)$ such that $U_{\theta\theta} V_{\theta\theta} = \mathbf{I}_\theta$. It follows that
   \[ A \cdot \diag(V_{\theta\theta}, \mathbf{I}_{l-\theta}) =  \begin{pmatrix}\begin{smallmatrix}
     h_1\varphi_r^{s_1} &   &   & h_1\varphi_r^{s_1+1} u_{1,\theta+1} & \cdots & h_1\varphi_r^{s_1+1} u_{1l} \\ & \ddots &  & \vdots & \ddots & \vdots \\
      &  & h_\theta\varphi_r^{s_\theta} & h_\theta\varphi_r^{s_\theta+1} u_{\theta,\theta+1} & \cdots & h_\theta\varphi_r^{s_\theta+1} u_{\theta l} \\
     h_{\theta+1}\varphi_r^{s}v_{\theta+1,1} & \cdots & h_{\theta+1}\varphi_r^{s}v_{\theta+1,\theta} & h_{\theta+1}\varphi_r^{s+1} u_{\theta+1,\theta+1} & \cdots & h_{\theta+1}\varphi_r^{s+1} u_{\theta+1,l} \\ \vdots & \ddots & \vdots & \vdots & \ddots & \vdots \\
     h_l\varphi_r^{s}v_{l1} & \cdots & h_l\varphi_r^{s}v_{l\theta} & h_l\varphi_r^{s+1} u_{l,\theta+1} & \cdots & h_l\varphi_r^{s+1} u_{ll}
    \end{smallmatrix}\end{pmatrix},\]
    where $v_{ij}\in \R_r$ for $\theta+1\leq i \leq l$ and $1\leq j \leq \theta$. Clearly, $\diag(V_{\theta\theta}, \mathbf{I}_{l-\theta})\in \GL_l(\R_r)$. By performing finitely many elementary row and column operations over $\R$ on $A \cdot \diag(V_{\theta\theta}, \mathbf{I}_{l-\theta})$, we obtain
  \begin{equation}\label{501-lemma-equ-7}
   A\sim_{\R} \begin{pmatrix}
     h_1\varphi_r^{s_1} &   &   &   &   &   \\ & \ddots &  &   &   &  \\
      &  & h_\theta\varphi_r^{s_\theta} &   &   &   \\
      &   &   & h_{\theta+1}\varphi_r^{s+1} v_{\theta+1,\theta+1} & \cdots & h_{\theta+1}\varphi_r^{s+1} v_{\theta+1,l} \\
         &   &   & \vdots & \ddots & \vdots \\
       &   &   & h_l\varphi_r^{s+1} v_{l,\theta+1} & \cdots & h_l\varphi_r^{s+1} v_{ll}
    \end{pmatrix} \triangleq A_1,
  \end{equation}
    where $v_{ij}\in \R_r$ for $\theta+1\leq i,j \leq l$. Let
    \[B = \begin{pmatrix}
         v_{\theta+1,\theta+1} & v_{\theta+1,\theta+2} & \cdots &  v_{\theta+1,l} \\
         \frac{h_{\theta+2}}{h_{\theta+1}}v_{\theta+2,\theta+1} & \frac{h_{\theta+2}}{h_{\theta+1}}v_{\theta+2,\theta+2} & \cdots & \frac{h_{\theta+2}}{h_{\theta+1}}v_{\theta+2,l} \\
         \vdots & \vdots & \ddots & \vdots \\
        \frac{h_{l}}{h_{\theta+1}} v_{l,\theta+1} & \frac{h_{l}}{h_{\theta+1}} v_{l,\theta+2} &\cdots & \frac{h_{l}}{h_{\theta+1}} v_{ll}
    \end{pmatrix}.\]
  Then it follows from Equation \eqref{501-lemma-equ-7} that
  \begin{equation}\label{501-lemma-equ-8}
   A_1 = \diag(h_1\varphi_r^{s_1},\ldots,h_\theta\varphi_r^{s_\theta}, h_{\theta+1}\varphi_r^{s+1},\ldots,h_{\theta+1}\varphi_r^{s+1}) \cdot \diag(\mathbf{I}_\theta, B).
  \end{equation}
  Since $A\sim_{\R} A_1$, by Lemma \ref{422-lemma-1} we obtain
   \[d_i(B) = \frac{h_{\theta+1}\cdots h_{\theta+i}}{h_{\theta+1}^i} ~ \text{and} ~ \J_i(B) = \J_{\theta+i}(A)=\R ~ \text{for} ~ i=1,\ldots,l-\theta.\]
  By the induction hypothesis, we have
  \begin{equation}\label{501-lemma-equ-9}
   B\sim_{\R} \diag(1,\frac{h_{\theta+2}}{h_{\theta+1}},\ldots,\frac{h_{l}}{h_{\theta+1}}).
  \end{equation}
  Combining Equations \eqref{501-lemma-equ-8} and \eqref{501-lemma-equ-9}, we have $A\sim_{\R} S$.

  \textbf{Third Case:} there exists an integer $\tau$ with $1\leq \tau \leq \theta-1$ such that
  \[s_1\leq \cdots \leq s_{\tau} < s_{\tau+1}  = \cdots = s_\theta = s.\]

  Let $U_\theta = (u_{ij})_{l\times l}$, where $u_{ij}\in \R_r$ for $1\leq i,j\leq l$. Then
  \[ A = \begin{pmatrix}
     h_1\varphi_r^{s_1}u_{11} & \cdots & h_1\varphi_r^{s_1}u_{1\theta} & h_1\varphi_r^{s_1+1}u_{1,\theta+1} & \cdots & h_1\varphi_r^{s_1+1}u_{1l} \\
     \vdots & \ddots & \vdots & \vdots & \ddots & \vdots \\
     h_{\tau}\varphi_r^{s_\tau}u_{\tau1} & \cdots & h_{\tau}\varphi_r^{s_\tau}u_{\tau \theta} & h_{\tau}\varphi_r^{s_\tau+1}u_{\tau,\theta+1} & \cdots & h_{\tau}\varphi_r^{s_\tau+1}u_{\tau l} \\
     h_{\tau+1}\varphi_r^{s}u_{\tau+1,1} & \cdots & h_{\tau+1}\varphi_r^{s}u_{\tau+1,\theta} & h_{\tau+1}\varphi_r^{s+1}u_{\tau+1,\theta+1} & \cdots & h_{\tau+1}\varphi_r^{s+1}u_{\tau+1,l} \\
     \vdots & \ddots & \vdots & \vdots & \ddots & \vdots \\
     h_l\varphi_r^{s}u_{l1} & \cdots & h_l\varphi_r^{s}u_{l\theta} & h_l\varphi_r^{s+1}u_{l,\theta+1} & \cdots & h_l\varphi_r^{s+1}u_{ll}
  \end{pmatrix}.\]
  Let $U_{\tau \theta}$ be the $\tau\times \theta$ submatrix formed by the first $\tau$ rows and the first $\theta$ columns of $U_\theta$. We assert that $U_{\tau \theta}$ is a ZLP matrix. If otherwise, there exists $\vec{\omega}=(\omega_1,\ldots,\omega_{r-1}, \omega_{r+1},\ldots,\omega_n)\in \mathds{V}(\I_{\tau}(U_{\tau \theta}))$. Let $\mathcal{C}_1$ be the set formed by all $\tau\times \tau$ minors of $U_{\tau \theta}$, and let $\mathcal{C}_2$ be the set formed by all $\tau\times \tau$ reduced minors of $A$. Then $\langle \mathcal{C}_1 \rangle_{\R_r} = \I_{\tau}(U_{\tau \theta})$ and $\langle \mathcal{C}_2 \rangle_{\R} = \J_{\tau}(A)$. Since $d_{\tau}(A) = h_1\cdots h_{\tau}\varphi_r^{s_1+\cdots+s_{\tau}}$, $\mathcal{C}_1$ is a part of $\mathcal{C}_2$. For any $q\in \mathcal{C}_2 \setminus \mathcal{C}_1$, it follows from $s_{\tau} < s$ that $\varphi_r\mid q$. Set $\omega_r = f_r(\omega_1,\ldots,\omega_{r-1})$. Then $q(\omega_1,\ldots,\omega_{r-1},\omega_{r},\omega_{r+1},\ldots,\omega_n)=0$, and \[(\omega_1,\ldots,\omega_{r-1},\omega_{r},\omega_{r+1},\ldots,\omega_n)
  \in\mathds{V}(\J_{\tau}(A)).\]
  This contradicts the fact that $\J_{\tau}(A) = \R$. Thus $U_{\tau \theta}$ is a ZLP matrix. According to the Quillen-Suslin theorem, there exists $V_{\theta\theta}\in \GL_{\theta}(\R_r)$ such that $U_{\tau \theta} V_{\theta\theta} = (\mathbf{I}_{\tau}, 0_{\tau\times (\theta-\tau)})$. By performing finitely many elementary row and column operations over $\R$ on $A\cdot \diag(V_{\theta\theta}, \mathbf{I}_{l-\theta})$, we obtain
  \[A \sim_{\R} \diag(h_1\varphi_r^{s_1},\ldots,h_{\tau}\varphi_r^{s_\tau}, h_{\tau+1}\varphi_r^{s},\ldots,h_{\tau+1}\varphi_r^{s})\cdot \diag(\mathbf{I}_{\tau},B) \triangleq A_1,\]
  where
  \[B = \begin{pmatrix} v_{\tau+1,\tau+1} & \cdots & v_{\tau+1,\theta} & \varphi_rv_{\tau+1,\theta+1} &  \cdots & \varphi_rv_{\tau+1,l} \\
  \frac{h_{\tau+2}}{h_{\tau+1}}v_{\tau+2,\tau+1} & \cdots & \frac{h_{\tau+2}}{h_{\tau+1}}v_{\tau+2,\theta} & \frac{h_{\tau+2}}{h_{\tau+1}}\varphi_rv_{\tau+2,\theta+1} &  \cdots & \frac{h_{\tau+2}}{h_{\tau+1}}\varphi_rv_{\tau+2,l} \\
  \vdots & \ddots & \vdots & \vdots & \ddots & \vdots \\
  \frac{h_{l}}{h_{\tau+1}}v_{l,\tau+1} & \cdots & \frac{h_{l}}{h_{\tau+1}}v_{l,\theta} & \frac{h_{l}}{h_{\tau+1}}\varphi_rv_{l,\theta+1} &  \cdots & \frac{h_{l}}{h_{\tau+1}}\varphi_rv_{ll}
  \end{pmatrix},\]
  $v_{ij}\in \R_r$ for $\tau+1\leq i,j \leq l$. By Lemma \ref{422-lemma-1}, we have $\J_i(B)=\R$ for $i=1,\ldots,l-\tau$, and
  \[d_i(B) = \begin{cases}
               \frac{h_{\tau+1}\cdots h_{\tau+i}}{h_{\tau+1}^i}, & i= 1, \ldots, \theta-\tau; \\
               \frac{h_{\tau+1}\cdots h_{\tau+i}\varphi_r^{i-(\theta-\tau)}}{h_{\tau+1}^i}, & i = \theta-\tau+1,\ldots,l-\tau.
             \end{cases}
  \]
  It follows that the Smith normal form of $B$ is
  \begin{equation*}\label{501-lemma-equ-10}
   S_B = \diag(1, \frac{h_{\tau+2}}{h_{\tau+1}},\ldots,\frac{h_{\theta}}{h_{\tau+1}} ,\frac{h_{\theta+1}}{h_{\tau+1}}\varphi_r, \ldots, \frac{h_{l}}{h_{\tau+1}} \varphi_r).
  \end{equation*}
  Let $V_{l-\tau}= (v_{ij})\in \M_{(l-\tau)\times (l-\tau)}(\R_r)$, where $\tau+1\leq i,j \leq l$. Then
  \begin{equation*}\label{501-lemma-equ-11}
   B = \diag(1, \frac{h_{\tau+2}}{h_{\tau+1}},\ldots,\frac{h_{\theta}}{h_{\tau+1}} ,\frac{h_{\theta+1}}{h_{\tau+1}}, \ldots, \frac{h_{l}}{h_{\tau+1}})\cdot V_{l-\tau} \cdot \diag(\underbrace{1,\ldots,1}_{\theta-\tau},\varphi_r,\ldots,\varphi_r).
  \end{equation*}
  It is easy to verify that $V_{l-\tau}\in \GL_{l-\tau}(\R_r)$. Adopting the proof technique from the \textbf{First Case}, we obtain $B\sim_{\R} S_B$. Therefore, $A\sim_{\R} S$. \rulex


\begin{thebibliography}{199}
 \bibitem{Cox2007}
  Cox D, Little J, and O'shea D, {\it Ideals, Varieties, and Algorithms}, Third Edition, Undergraduate Texts in Mathematics, Springer, New York, 2007.

 \bibitem{Noferini2025Smith}
  Noferini V and Williams G, Smith forms of matrices in companion rings, with group theoretic and topological applications, {\it Linear Algebra and Its Applications}, 2025, {\bf 708}: 372–404.

 \bibitem{Bose1982}
  Bose N, {\it Applied Multidimensional Systems Theory}, Van Nostrand Reinhold, New York, 1982.

 \bibitem{Bose2003}
  Bose N (ed.), with contributions by Bose N, Buchberger B, and Guiver J, {\it Multidimensional Systems Theory and Applications}, Dordrecht, The Netherlands: Kluwer, 2003.

 \bibitem{Gohberg1982Matrix}
  Gohberg I, Lancaster P, and Rodman L, {\it Matrix Polynomials}, Academic Press, New York, 1982.

 \bibitem{Frost1978}
  Frost M and Storey C, Equivalence of a matrix over $R[s,z]$ with its Smith form, {\it International Journal of Control}, 1978, {\bf 28}(5): 665-671.

 \bibitem{Frost1981}
  Frost M and Storey C, Equivalence of matrices over $R[s,z]$: a counter-example, {\it International Journal of Control}, 1981, {\bf 34}(6): 1225-1226.

 \bibitem{Lin2006On}
  Lin Z, Boudellioua M, and Xu L, On the equivalence and factorization of multivariate polynomial matrices, {\it Proceedings of 2006 IEEE International Symposium on Circuits and Systems (ISCAS 2006)}, Greece, 2006, 4911-4914.

 \bibitem{Liu2024}
  Liu J, Li D, and Wu T, The Smith normal form and reduction of weakly linear matrices, {\it Journal of Symbolic Computation}, 2024, {\bf 120}: 102232.

 \bibitem{Liu2025The}
  Liu J, Wu T, Guan J, and Kang Y, The Smith form of quasi weakly linear polynomial matrices, {\it Journal of Systems Science and Mathematical Sciences}, 2026, {\bf 46}(8): 2490-2506.

 \bibitem{LiD2019}
  Li D, Liang R, and Liu J, Some further results on the Smith form of bivariate polynomial matrices, {\it Journal of Systems Science and Mathematical Sciences}, 2019, {\bf 39}(12): 1983-1997.

 \bibitem{Zheng2023New}
  Zheng X, Lu D, Wang D, and Xiao F, New results on the equivalence of bivariate polynomial matrices, {\it Journal of Systems Science and Complexity}, 2023, {\bf 36}(1): 77-95.

 \bibitem{Guan2025NewR}
  Guan J, Liu J, Zheng L, Wu T, and Liu J, New results on equivalence of multivariate polynomial matrices, {\it Journal of Systems Science and Complexity}, 2025, {\bf 38}(4): 1823-1832.

 \bibitem{Lu2024arxiv}
  Lu D, Wang D, Xiao F, and Zheng X, On the equivalence problem of Smith forms for multivariate polynomial matrices, 2024, arXiv:2407.06649

 \bibitem{Zheng2024Biv}
  Zheng L, Wu T, and Liu J, Bivariate polynomial matrix and Smith form, {\it Mathematics}, 2024, {\bf 12}: 815.

 \bibitem{Lu2026Smith}
  Lu D, Wang D, Xiao F, and Zheng X, Smith normal forms of bivariate polynomial matrices, {\it Linear Algebra and Its Applications}, 2026, {\bf 737}: 308-330.

 \bibitem{Li2025Smith}
  Li D, Chen Z, and Liu J, Smith form of the matrix over the unique factorization domain, {\it Scientia Sinica Mathematica}, 2025, {\bf 55}: 919-926.

 \bibitem{Lin1988}
  Lin Z, On matrix fraction descriptions of multivariable linear n-D systems, {\it IEEE Transactions on Circuits and Systems}, 1988, {\bf 35}(10): 1317-1322.

 \bibitem{Strang2010Linear}
  Strang G, {\it Linear Algebra and Its Applications}, Academic Press, New York, 2010.

 \bibitem{Li2022Li}
  Li D, Liu J, and Chu D, The Smith form of a multivariate polynomial matrix over an arbitrary coefficient field, {\it Linear and Multilinear Algebra}, 2022, {\bf 70}(2): 366-379.

 \bibitem{Youla1979Notes}
  Youla D and Gnavi G, Notes on n-dimensional system theory, {\it IEEE Transactions on Circuits and Systems}, 1979, {\bf 26}(2): 105-111.

 \bibitem{Quillen1976Projective}
  Quillen D, Projective modules over polynomial rings, {\it Inventiones Mathematicae}, 1976, {\bf 36}(1): 167-171.

 \bibitem{Suslin1976Projective}
  Suslin A, Projective modules over polynomial rings are free, {\it Soviet Mathematics - Doklady}, 1976, {\bf 17}(4): 1160-1164.

 \bibitem{Serre1955}
  Serre J, Faisceaux alg\'{e}briques coh\'{e}rents, {\it Annals of Mathematics}, 1955, {\bf 61}: 197–278.

 \bibitem{Lam1978}
  Lam T, {\it Serre's Conjecture}, Lecture Notes in Mathematics 635, Springer, Berlin, 1978.

 \bibitem{Wang2004On}
  Wang M and Feng D, On Lin-Bose problem, {\it Linear Algebra and Its Applications}, 2004, {\bf 390}(1): 279-285.

 \bibitem{Lin2001A}
  Lin Z and Bose N, A generalization of Serre's conjecture and some related issues, {\it Linear Algebra and Its Applications}, 2001, {\bf 338}: 125-138.

 \bibitem{Lu2025JSSC}
  Lu D, Wang D, Xiao F, and Zheng X, Theory of Smith forms for bivariate polynomial matrices, {\it Journal of Systems Science and Complexity}, 2025, DOI: 10.1007/s11424–025–5125–0

 \bibitem{Guiver1982}
  Guiver J and Bose N, Polynomial matrix primitive factorization over arbitrary coefficient field and related results, {\it IEEE Transactions on Circuits and Systems}, 1982, {\bf 29}(10): 649-657.

 \bibitem{Zeng2025Poly}
  Zeng Z, Liu J, and Wu T, Polynomial algebra isomorphism and Smith form of matrices, {\it Journal of Systems Science and Mathematical Sciences}, 2026, {\bf 46}(8): 2480-2489.

 \bibitem{Essen2000Polynomial}
  van den Essen A, {\it Polynomial Automorphisms and the Jacobian Conjecture}, volume 190 of Progress in Mathematics, Birkh\"{a}user, Basel, 2000.

 \bibitem{jung1942ganze}
  Jung H, \"{U}ber ganze birationale Transformationen der Ebene, {\it Journal f\"{u}r die reine und angewandte Mathematik}, 1942, {\bf 184}: 161-174.

 \bibitem{vanderkulk1953polynomial}
  van der Kulk W, On polynomial rings in two variables, {\it Nieuw Archief voor Wiskunde}, 1953, {\bf 1}: 33-41.

 \bibitem{nagata1972autom}
  Nagata M, {\it On Automorphism Group of $k[x,y]$}, volume 5 of Lectures in Mathematics, Kyoto University, Kinokuniya, Tokyo, 1972.

 \bibitem{shestakov2003nagata}
  Shestakov I and Umirbaev U, The Nagata automorphism is wild, {\it Proceedings of the National Academy of Sciences}, 2003, {\bf 100}: 12561-12563.

 \bibitem{shestakov2004tame}
  Shestakov I and Umirbaev U, The tame and the wild automorphisms of polynomial rings in three variables, {\it Journal of the American Mathematical Society}, 2004, {\bf 17}: 197-227.

 \bibitem{Fabianska2007}
  Fabia\'{n}ska A and Quadrat A, Applications of the Quillen-Suslin theorem to multidimensional systems theory, {\it Radon Series on Computational and Applied Mathematics}, 2007, {\bf 3}: 23-106.
\end{thebibliography}
\end{document}